\documentclass[11pt]{amsart}

\usepackage{graphicx}
\usepackage{amssymb}
\usepackage{braket,tensor,bm,stmaryrd,fouridx}
\usepackage{tikz-cd}
\usepackage{mathtools}\mathtoolsset{showonlyrefs}
\numberwithin{equation}{section}
\usepackage{stackengine}
\usepackage{rotating}
\usepackage{mathrsfs}

\newcommand{\abs}[1]{\lvert#1\rvert}

\newcommand{\conj}[1]{\overline{#1}}

\DeclareMathOperator{\tf}{tf}
\DeclareMathOperator{\Ad}{Ad}

\DeclareMathOperator{\Ker}{Ker}
\DeclareMathOperator{\End}{End}
\DeclareMathOperator{\Hom}{Hom}

\DeclareMathOperator{\vol}{vol}

\DeclareMathOperator{\Ric}{Ric}

\DeclareMathOperator{\tr}{tr}

\newcommand{\SU}{\mathit{SU}}
\newcommand{\PSU}{\mathit{PSU}}

\newcommand{\GL}{\mathit{GL}}

\newcommand{\intprod}{\mathbin{\raisebox{1pt}{\scalebox{1.3}{$\lrcorner$}}}}

\newcommand{\transpose}[1]{\fourIdx{t}{}{}{}{#1}}

\DeclareMathOperator{\II}{I\hspace{-1.2pt}I}
\DeclareMathOperator{\conn}{Conn}
\DeclareMathOperator{\curv}{Curv}
\DeclareMathOperator{\Res}{Res}

\DeclareMathOperator{\ch}{ch}

\usepackage{amsthm}
\newtheorem{thm}{Theorem}[section]
\newtheorem{prop}[thm]{Proposition}
\newtheorem{lem}[thm]{Lemma}
\newtheorem{cor}[thm]{Corollary}
\theoremstyle{definition}
\newtheorem{dfn}[thm]{Definition}
\theoremstyle{remark}
\newtheorem{rem}[thm]{Remark}
\newtheorem{ex}[thm]{Example}

\usepackage{hyperref}
\title{Secondary Characteristic Classes in CR Geometry}
\author{Shuya Matsumoto}

\begin{document}

\maketitle

\begin{abstract}
    We construct a family of $\mathbb{R}/\mathbb{Z}$-valued CR invariants for compact strictly pseudoconvex CR manifolds admitting pseudo-Einstein contact forms. 
    We define these invariants by applying the theory of Cheeger--Simons differential characters to a globally defined modification of the normal tractor connection. 
    When the CR holomorphic tangent bundle is trivial, their natural $\mathbb{R}$-valued lifts agree with the generalized Burns--Epstein invariants. 
    For CR manifolds arising as boundaries of relatively compact strictly pseudoconvex domains, we identify these differential characters with those determined by the renormalized connection and derive bulk--boundary formulas involving renormalized characteristic numbers and residues of Chern classes. 
    These formulas recover and extend the results of Burns--Epstein and Marugame. 
    We also obtain obstructions to CR embeddings into complex Euclidean space.
\end{abstract}

\tableofcontents

\section{Introduction}
This paper constructs a family of secondary characteristic invariants associated with the normal Cartan connection on compact strictly pseudoconvex CR manifolds admitting pseudo-Einstein contact forms. 
More precisely, we apply the theory of Cheeger--Simons differential characters to a globally defined modification of the normal tractor connection induced by the normal Cartan connection.

The first invariant of this type was introduced by Burns and Epstein \cite{burns_global_1988} for compact three-dimensional strictly pseudoconvex CR manifolds whose CR holomorphic tangent bundle is trivial. 
Their construction proceeds as follows. 
Let $s$ be a global section of the Cartan bundle; such a section exists because $T^{1,0}M$ is trivial. 
Let $\omega$ denote the normal Cartan connection.
The \emph{Burns--Epstein invariant} is defined by integrating the Chern--Simons form associated with the second Chern polynomial $c_2$:
\begin{equation}
\widetilde{\mu}^{\rm BE}_{c_2}(M)
= \int_M s^*T_{c_2}(\omega).
\end{equation}
This invariant also admits a local expression in terms of a chosen pseudohermitian structure. 
Burns and Epstein further showed that the invariant is stationary with respect to all deformations of the CR structure if and only if the CR structure is spherical, that is, if and only if the curvature of the normal Cartan connection vanishes; see also \cite{cheng_burns-epstein_1990}.

Subsequently, Burns and Epstein \cite{burns_characteristic_1990} defined analogous invariants for the boundaries of bounded strictly pseudoconvex domains in $\mathbb{C}^{n+1}$. 
Although the Cartan bundle need not be trivial, an associated vector bundle of rank $n+2$, now known as the \emph{standard tractor bundle}, is trivial in this setting. 
Their construction can be interpreted as using this triviality to define the corresponding Chern--Simons invariants.

They also related these boundary invariants to biholomorphic invariants of the domain. 
More precisely, they modified the Chern connection of a complete Kähler--Einstein metric $g$ on $\Omega$ by subtracting its divergent part. 
The resulting connection $\conj{\nabla}^g$, called the \emph{renormalized connection}, has characteristic forms that are integrable over the domain and whose integrals are biholomorphically invariant.
They then proved that
\begin{align}
\int_\Omega c_{n+1}(\conj{\nabla}^g)
&= \chi(\Omega) + \widetilde{\mu}_{c_{n+1}}(\partial\Omega),\\
\int_\Omega (c_{q_1}\cdots c_{q_m})(\conj{\nabla}^g)
&= \widetilde{\mu}_{c_{q_1}\cdots c_{q_m}}(\partial\Omega),
\quad (q_1+\cdots+q_m=n+1,m\ge 2).
\end{align}

For a general compact strictly pseudoconvex CR manifold, the standard tractor bundle may fail to exist globally, and even when it does exist, it need not be trivial. 
Thus, the preceding construction using an ordinary Chern--Simons form does not directly apply.

Marugame \cite{marugame_renormalized_2016} addressed this difficulty by using a relative Chern--Simons form associated with the renormalized connection $\conj{\nabla}^g$.
Let $\Omega$ be a relatively compact strictly pseudoconvex domain in a complex manifold. 
To construct the renormalized connection, one assumes that the canonical bundle of $\partial\Omega$ carries a flat Hermitian metric. 
Such a metric always exists when $\Omega$ is a bounded strictly pseudoconvex domain in $\mathbb{C}^{n+1}$ and is called a \emph{pseudo-Einstein structure} \cite{lee_pseudo-einstein_1988,hirachi_variation_2017,hislop_cr-invariants_2006}. 
Since Hermitian metrics on the canonical bundle correspond to contact forms, we call the contact form corresponding to a pseudo-Einstein structure a \emph{pseudo-Einstein contact form}.

Fix a $(1,0)$-vector field $\xi$ defined near the boundary whose real part is transverse to $\partial\Omega$ and points outward, and let $\nabla^\xi$ be a connection on $T^{1,0}\conj{\Omega}|_{\partial\Omega}$ satisfying $\nabla^\xi\xi=0$.
Marugame then defined the global CR invariant
\begin{equation}
\widetilde{\mu}_{c_{n+1}}(\partial\Omega) = \int_{\partial\Omega} c_{n+1}(\nabla^\xi,\conj{\nabla}^g),
\end{equation}
and proved the \emph{renormalized Chern--Gauss--Bonnet formula}
\begin{equation}
\int_\Omega c_{n+1}(\conj{\nabla}^g) = \chi(\Omega) + \widetilde{\mu}_{c_{n+1}}(\partial\Omega).
\end{equation}

This paper addresses the above difficulty by using pseudo-Einstein structures and Cheeger--Simons differential characters. 
A pseudo-Einstein structure is used to modify the normal Cartan connection so that it induces a connection on a globally defined vector bundle.

More precisely, we tensor the standard tractor bundle $\mathcal{T}^\#$, which may exist only locally, with the line bundle $\mathcal{E}(1,0)$, the dual of an $(n+2)$nd root of the canonical bundle. 
The resulting vector bundle $\mathcal{T}=\mathcal{E}(1,0)\otimes\mathcal{T}^\#$ is globally well-defined. 
Let $\nabla^{\mathcal{T}^\#}$ denote the connection on $\mathcal{T}^\#$ induced by the normal Cartan connection $\omega$.
The normal Cartan connection alone, however, does not canonically determine a connection on $\mathcal{E}(1,0)$, and hence an additional choice of connection is required. 
We use the flat connection $D^\theta$ on $\mathcal{E}(1,0)$ determined by a pseudo-Einstein contact form $\theta$. 
Since $D^\theta$ is flat, the curvature of the tensor product connection
\begin{equation}
\nabla^{\mathcal{T}}=D^\theta\otimes\nabla^{\mathcal{T}^\#}
\end{equation}
is naturally identified with that of $\nabla^{\mathcal{T}^\#}$.

We then apply the theory of Cheeger--Simons differential characters \cite{chern_characteristic_1974,cheeger_differential_1985} to this tensor product connection. 
A Cheeger--Simons differential character is an $\mathbb{R}/\mathbb{Z}$-valued homomorphism on smooth cycles whose values on boundaries are prescribed by a characteristic form.

In the concrete description used in this paper, such a character is represented by integrating a relative Chern--Simons form over a lift of a cycle to an appropriate Stiefel bundle. Different choices of the lift may change the integral by an integer. Reduction modulo $\mathbb{Z}$ therefore gives a well-defined differential character.

When a suitable global partial frame is available, it determines a section of the Stiefel bundle and hence removes the ambiguity in choosing a lift of the cycle. The differential character then admits an $\mathbb{R}$-valued lift represented by the corresponding relative Chern--Simons form. This lift may depend on the homotopy class of the partial frame, and the difference between the lifts associated with two partial frames is described explicitly by obstruction theory; see Section 3.2.

For the tractor bundle $\mathcal{T}$, there are several natural situations in which the associated Cheeger--Simons differential characters admit such $\mathbb{R}$-valued lifts.
Our main claim is the following:

\emph{The Cheeger--Simons differential characters associated with $\nabla^{\mathcal{T}}=D^\theta\otimes\nabla^{\mathcal{T}^\#}$ define $\mathbb{R}/\mathbb{Z}$-valued CR invariants that are independent of the choice of pseudo-Einstein contact form $\theta$. 
In several situations, these invariants admit canonical or frame-dependent $\mathbb{R}$-valued lifts. 
The invariants introduced by Burns--Epstein and Marugame arise as such lifts.}

The $\mathbb{R}/\mathbb{Z}$-valued CR invariants are defined in Section 3. We then show that, when the CR holomorphic tangent bundle is trivial, their natural $\mathbb{R}$-valued lifts agree with the generalized Burns--Epstein invariants. 
In dimension three, this recovers the original Burns--Epstein invariant.

In Section 4, we show that, when $M$ arises as the boundary of a strictly pseudoconvex domain, the differential characters associated with $\nabla^{\mathcal{T}}$ agree with those arising from the renormalized connection $\conj{\nabla}^g$. 
We prove this in two steps. 
First, we intrinsically reduce $\nabla^{\mathcal{T}}$ to a connection on a rank-$(n+1)$ subbundle of $\mathcal{T}$. 
Second, in the boundary case, we identify the reduced connection with the renormalized connection through an ambient construction.

\begin{thm}[$=\text{Theorem}$ \ref{thm:reduction_theorem}]
    There exists a rank-$(n+1)$ subbundle $\underline{\mathcal{T}}\subset\mathcal{T}$ equipped with a connection $\nabla^{\underline{\mathcal{T}}}$ such that the Cheeger--Simons differential characters associated with $\nabla^{\underline{\mathcal{T}}}$ agree with those associated with $\nabla^{\mathcal{T}}$.
\end{thm}

\begin{thm}[$=\text{Theorem}$ \ref{thm:ambient_construction}]
    Suppose that the CR manifold arises as the boundary of a strictly pseudoconvex domain $\Omega$. 
    Then there exists a rank-$(n+2)$ vector bundle $\widetilde{\mathcal{T}}\to\conj{\Omega}$ equipped with a metric connection $\nabla^{\widetilde{h}}$ such that:
    \begin{enumerate}
        \item $\widetilde{\mathcal{T}}$ contains $T^{1,0}\conj{\Omega}$ as a subbundle;
        \item the connection on $T^{1,0}\conj{\Omega}$ obtained by projecting $\nabla^{\widetilde{h}}$ onto this subbundle coincides with the renormalized connection $\conj{\nabla}^g$; and
        \item there exists an isomorphism of pairs of vector bundles $(\widetilde{\mathcal{T}}|_{\partial\Omega},T^{1,0}\conj{\Omega}|_{\partial\Omega})\simeq(\mathcal{T},\underline{\mathcal{T}})$ under which the pair of connections $(\nabla^{\widetilde{h}},\conj{\nabla}^g)$ corresponds to $(\nabla^{\mathcal{T}},\nabla^{\underline{\mathcal{T}}})$.
    \end{enumerate}
\end{thm}

Combining these two theorems with the Stokes formula for differential characters and the residue formula developed in Section 4, we obtain bulk--boundary formulas expressing our invariants in terms of renormalized characteristic numbers. 
When an $\mathbb{R}$-valued lift is defined by a partial frame, the correction term is given by the pairing of the corresponding Chern-class residue with a complementary characteristic class. 
These formulas recover and extend the formulas of Burns--Epstein \cite{burns_characteristic_1990} and Marugame \cite{marugame_renormalized_2016}.

In Section 5, we compute our invariants for two classes of examples: boundaries of Reinhardt domains and regular Sasakian $\eta$-Einstein manifolds.
The Reinhardt-domain example gives a family along which a natural $\mathbb{R}$-valued lift varies smoothly. 
For example, for the boundaries of 
\begin{equation}
    \textstyle
    \Omega_r = \left\{(z_1,\ldots,z_4) \in (\mathbb{C}^{\times})^4 \mid \sum_{i=1}^4 (\log \abs{z_i})^2 < r^2 \right\},
\end{equation}
we have 
\begin{equation}
        \widetilde{\mu}_{c_2\cdot c_2}(\partial \Omega_r) = \frac{75\pi^2}{256r^4}.
\end{equation}
The computation for regular Sasakian $\eta$-Einstein manifolds, together with the Rossi-sphere computation, shows that the invariants associated with $c_1\cdot c_1$ and $c_2$ are not universally proportional in dimension three.

Our secondary invariants also yield obstructions to CR embeddings. 
More precisely, the nonvanishing of an invariant associated with a Chern monomial of the form $c_1\cdot\Phi$ obstructs the existence of a partial frame of $\mathcal{T}$ of the required rank. 
Since $\mathcal{T}$ is trivial for every CR hypersurface in $\mathbb{C}^{n+1}$, such nonvanishing obstructs a CR embedding into $\mathbb{C}^{n+1}$; see Proposition \ref{thm:obstruction_to_embedding} for a precise statement.
For the standard CR lens space $\mathbb{S}^3/\mathbb{Z}_m$, $m\ge2$, which does not admit even a topological embedding into $\mathbb{C}^2$, we obtain
\begin{equation}
    \mu_{c_1\cdot c_1}(\mathbb{S}^3/\mathbb{Z}_m) = -\frac{4}{m}+\mathbb{Z}.
\end{equation}

Finally, several other global CR invariants have been defined for compact strictly pseudoconvex CR manifolds admitting pseudo-Einstein structures. 
These include the total $Q'$-curvature \cite{case_paneitz-type_2013,hirachi_q-prime_2014}, the total $\mathcal{I}'$-curvatures \cite{case_p-operator_2020,marugame_renormalized_2021,case_cal_2023}, and the global CR invariants via renormalized characteristic forms \cite{marugame_renormalized_2021}. 
The first two are special cases of the last.

Takeuchi \cite{takeuchi_formulae_2024} conjectured that the Burns--Epstein invariant associated with the top Chern polynomial lies in the linear span of the global CR invariants via renormalized characteristic forms, and proved this conjecture for $n=1,2$. 
In dimension three, there is, up to normalization, only one global CR invariant via renormalized characteristic forms, and Takeuchi's result shows that it is a constant multiple of $\widetilde{\mu}_{c_2}$. 
By contrast, our computations show that $\widetilde{\mu}_{c_1\cdot c_1}$ is not proportional to $\widetilde{\mu}_{c_2}$. 
Thus, $\widetilde{\mu}_{c_1\cdot c_1}$ cannot be expressed as a global CR invariant via renormalized characteristic forms.

\subsection*{Notation}
Throughout this paper, we use Penrose's abstract index notation \cite{penrose_spinors_1984}. 
Lowercase Greek indices $\alpha,\beta,\gamma,\delta,\rho,\sigma,\ldots$ range from $1$ to $n$; lowercase Roman indices $i,j,k,l,\ldots$ range from $1$ to $n+1$; and uppercase Roman indices $A,B,\ldots$ range from $0$ to $n+1$.

\subsection*{Acknowledgments}
The author is deeply grateful to his supervisor, Kengo Hirachi, for introducing him to this problem and for many valuable suggestions.
He would like to thank Taro Asuke for suggesting the use of differential characters. 
He is also indebted to Taiji Marugame for helpful comments in seminars and for pointing out several errors in an earlier version of this work. 
He thanks Tatsuo Suwa for valuable advice on residue theory. 
He also thanks Yuya Takeuchi for suggesting several examples and for helpful comments on Sasakian geometry. 
Finally, he thanks Yoshiaki Suzuki for useful feedback in seminars.

This work was supported by JSPS KAKENHI Grant Number JP26KJ0948 and by the World-leading Innovative Graduate Study for Frontiers of Mathematical Sciences and Physics (WINGS-FMSP), the University of Tokyo.

\section{CR Geometry}

In this section, we present the necessary definitions and results on CR structures.
The construction of the Cartan bundle is explained in some detail. 

\subsection{Basic definitions}
\label{subsec:Basic}
Let $M$ be a smooth manifold of dimension $2n+1$. 
A hyperplane distribution $HM \subset TM$ equipped with an almost complex structure $J \in \End{HM}$ is called an \emph{almost CR structure}.
It gives rise to a decomposition 
\begin{equation*}
    H^{\mathbb{C}}M = T^{1,0}M \oplus T^{0,1}M,
\end{equation*}
where $T^{1,0}M$ and $T^{0,1}M$ are the eigen-distributions of $J$ corresponding to the eigenvalues $\sqrt{-1}$ and $-\sqrt{-1}$, respectively.
An almost CR structure is an \emph{CR structure} if it is integrable in the sense that 
\begin{equation*}
    [T^{1,0}M,T^{1,0}M] \subset T^{1,0}M.
\end{equation*}
A smooth manifold equipped with a CR structure is called a CR manifold.
A typical example of a CR manifold is a real hypersurface $M$ of a complex manifold $X$ of complex dimension $n+1$, where the CR structure is given by $T^{1,0}M = T^{1,0}X|_M \cap T^{\mathbb{C}}M$.

Let $M$ be a CR manifold and assume that the bundle $TM/HM$ is oriented. 
Let $\theta$ be a positive global section of $(TM/HM)^* = HM^\perp$.
Then $\theta$ is a \emph{contact form}, i.e., $\theta \wedge (d\theta)^n \ne 0$ everywhere, if and only if the \emph{Levi form} 
\begin{equation*}
    l_\theta (Z,\conj{W}) \coloneqq -\sqrt{-1} d\theta (Z,\conj{W}) = \sqrt{-1} \theta([Z,\conj{W}]), \quad Z,W \in T^{1,0}M
\end{equation*}
is non-degenerate.
A choice of contact form is called a \emph{pseudohermitian} structure.
It follows that $l_{e^\Upsilon \theta} = e^\Upsilon l_\theta$ for a smooth real function $\Upsilon$.
We say that the CR manifold equipped with the prescribed orientation of $TM/HM$ is \emph{strictly pseudoconvex} if the Levi form of one, and hence every, positive contact form is positive definite.

The bundle of $(n+1)$-forms $\mathcal{K}_M=\wedge^{n+1}(T^{0,1}M)^\perp$ is called the \emph{canonical bundle} of $M$.
For integers $w_1,w_2 \in \mathbb{Z}$, define the density bundle by 
\begin{equation*}
    \mathcal{E}(w_1,w_2) = \mathcal{K}_M^{-w_1/(n+2)} \otimes (\conj{\mathcal{K}}_M)^{-w_2/(n+2)}.
\end{equation*}
Since the $(n+2)$nd root of $\mathcal{K}_M$ need not exist globally, the density bundle is intrinsic and globally defined if $w_1-w_2 \in (n+2)\mathbb{Z}$; otherwise, it is defined only locally.
Moreover, if $w_1=w_2\eqqcolon w$, the real part $\mathcal{E}_{\mathbb{R}}(w,w)$ is naturally defined and oriented.

\begin{lem}[cf. {\cite[Lemma 3.2]{lee_fefferman_1986}}]
    Let $M$ be a strictly pseudoconvex CR manifold. There is a canonical isomorphism between $\mathcal{E}_{\mathbb{R}}(1,1)$ and $TM/HM$.
\end{lem}

\begin{proof}
    First, we construct an isomorphism $\mathcal{E}_{\mathbb{R}}(-n-2,-n-2)\simeq(HM^\perp)^{\otimes(n+2)}$, for which only the nondegeneracy of the Levi form is required. 
    The Lie bracket followed by the natural projection defines the Levi bracket
    \begin{equation}
        T^{1,0}M\otimes T^{0,1}M\longrightarrow TM/HM.
    \end{equation}
    By nondegeneracy, taking its determinant gives an isomorphism
    \begin{equation}
        \wedge^nT^{1,0}M\otimes\wedge^nT^{0,1}M \longrightarrow (TM/HM)^{\otimes n}.
    \end{equation}
    After tensoring with $TM/HM$ and taking the appropriate duals, we obtain the desired isomorphism.

    Now assume that $M$ is strictly pseudoconvex. 
    The dual of the preceding isomorphism is $\mathcal{E}_{\mathbb{R}}(n+2,n+2) \simeq (TM/HM)^{\otimes(n+2)}$.
    Since both line bundles are naturally oriented, taking the unique orientation-preserving $(n+2)$nd root yields a canonical isomorphism $\mathcal{E}_{\mathbb{R}}(1,1)\simeq TM/HM$.
    This proves the lemma.
\end{proof}

As a corollary, if $M$ is strictly pseudoconvex, a choice of a contact form $\theta$ determines a metric $h_\theta \in \mathcal{E}_{\mathbb{R}}(-1,-1)$ on $\mathcal{E}(1,0)$.
By the construction of the isomorphism, it follows that a section $\zeta$ of $\mathcal{E}(1,0)$ has unit length with respect to $h_\theta$ if and only if 
\begin{equation}
    \theta \wedge (d\theta)^n = \sqrt{-1}^{n^2} n! \theta \wedge (T \intprod \zeta^{-n-2}) \wedge (T \intprod (\conj{\zeta})^{-n-2}),
    \label{eq:volume-renormalization}
\end{equation}
where $T$ is the unique vector field characterized by 
\begin{equation*}
    \theta (T) = 1, \quad T \intprod d\theta = 0.
\end{equation*}
This vector field is called the \emph{Reeb vector field} of the contact form $\theta$.

\subsection{The Tanaka--Webster connection}
Let $M$ be a strictly pseudoconvex CR manifold and fix a contact form $\theta$.
If $Z_\alpha$ is a local frame for $T^{1,0}M$, then $(T,Z_\alpha,Z_{\conj{\alpha}})$ is a local frame for $T^{\mathbb{C}}M$.
Let $\theta,\theta^\alpha,\theta^{\conj{\alpha}}$ be the dual coframe. 
By the integrability condition and the definition of $T$, we have 
\begin{equation}
    d\theta = \sqrt{-1} l_{\alpha \conj{\beta}} \theta^\alpha \wedge \theta^{\conj{\beta}}, \quad l_{\alpha \conj{\beta}} \coloneqq l_\theta (Z_\alpha,Z_{\conj{\beta}}).
    \label{eq:str.eq.forTW_1}
\end{equation}
If the frame $(Z_\alpha)$ is unitary with respect to the Levi form $l_\theta$, the coframe $\theta,\theta^\alpha,\theta^{\conj{\alpha}}$ is called an \emph{admissible coframe}.

The integrability condition also implies that there exist one-forms $\omega_{\beta}^{\,\,\,\alpha}$ and $\tau^\alpha = A^{\alpha}_{\,\,\,\conj{\beta}}\theta^{\conj{\beta}}$ satisfying 
\begin{equation}
    d\theta^\alpha = \theta^\beta \wedge \omega_{\beta}^{\,\,\,\alpha} + \theta \wedge \tau^\alpha,
    \label{eq:str.eq.forTW_2}
\end{equation}
which are uniquely determined by the condition 
\begin{equation*}
    \omega_{\alpha \conj{\beta}} + \omega_{\conj{\beta} \alpha} = 0,
\end{equation*}
where we raise and lower the indices using the metric $l_{\alpha \conj{\beta}}$.
The connection $\nabla^{TM}$ on $TM$ determined by the structure equations \eqref{eq:str.eq.forTW_1} and \eqref{eq:str.eq.forTW_2} is called the \emph{Tanaka--Webster connection} \cite{webster_pseudo-hermitian_1978}.
It induces connections on all density bundles $\mathcal{E}(w_1,w_2)$. 
For example, the connection form on $\mathcal{E}(-1,0)$ with respect to the frame $(\theta \wedge \theta^1 \wedge \cdots \wedge \theta^n)^{1/(n+2)}$ is given by $-\frac{1}{n+2}\omega_{\gamma}^{\,\,\,\gamma}$.

The second structure equation for the Tanaka--Webster connection reads  
\begin{align}
    d\omega_{\beta}^{\,\,\,\alpha} &= \omega_{\beta}^{\,\,\,\gamma} \wedge \omega_{\gamma}^{\,\,\,\alpha} + R_{\beta\,\,\, \rho \conj{\sigma}}^{\,\,\,\alpha} \theta^\rho \wedge \theta^{\conj{\sigma}}\\ &\quad+ A_{\beta \rho,}^{\quad\alpha} \theta^\rho \wedge \theta - A^{\alpha}_{\,\,\,\conj{\sigma},\beta} \theta^{\conj{\sigma}} \wedge \theta + \sqrt{-1} \theta_{\beta} \wedge \tau^\alpha - \sqrt{-1} \tau_\beta \wedge \theta^\alpha,
\end{align}
where the indices after the comma indicate the covariant derivatives with respect to $\nabla^{TM}$.
The tensor $R_{\beta\,\,\, \rho \conj{\sigma}}^{\,\,\,\alpha}$ is called the Tanaka--Webster curvature tensor.
Its Ricci and scalar curvatures are denoted by $R_{\alpha \conj{\beta}} = R_{\gamma\,\,\, \alpha \conj{\beta}}^{\,\,\,\gamma}$ and $R=R_{\gamma}^{\,\,\,\gamma}$, respectively.

\subsection{Groups appearing in CR Cartan geometry}
\label{subsec:Groups}
We collect some definitions and basic properties of the groups appearing in CR Cartan geometry.

Let $\mathbb{V}=(\mathbb{C}^{n+2},h)$ be a complex vector space equipped with a Lorentz--Hermitian metric 
\begin{equation}
    h = \begin{pmatrix}
     & & 1\\
     & \delta_{\alpha\conj{\beta}}& \\
    1& & 
    \end{pmatrix}.
\end{equation}
Let $P$ denote the subgroup of $\PSU(h)$ consisting of the elements that preserve the complex null line $\mathrm{span}\{\transpose{(1,0,\ldots,0)}\}$.
In matrix form, $P$ can be written as $P=P^\#/\mathbb{Z}_{n+2}$, where $P^\# \subset \SU(h)$ is the group of matrices of the form 
\begin{equation}
    \begin{pmatrix}
        c& -c u_{\beta}^{\,\,\,\gamma}v_\gamma& d\\
         & u_{\beta}^{\,\,\,\alpha}& v^\alpha\\
         & & \conj{c}^{-1}
    \end{pmatrix}, \quad u_\beta^{\,\,\,\alpha} \in \mathit{U}(n),\, c\conj{c}^{-1}\det u_\beta^{\,\,\,\alpha} =1,\,
    \left\lVert\begin{pmatrix}
        d\\ v^\alpha\\ \conj{c}^{-1}
    \end{pmatrix}\right\rVert_h =0.
    \label{eq:Psharp}
\end{equation}

Let $G_0^\#$ and $P^\#_+$ be the subgroups consisting of matrices of the following respective forms: 
\begin{equation*}
    G_0^\# = \left\{
    \begin{pmatrix}
        c& & \\
         & u_{\beta}^{\,\,\,\alpha}& \\
         & & \conj{c}^{-1}
    \end{pmatrix} \,\,\middle|\,\, u_\beta^{\,\,\,\alpha} \in \mathit{U}(n),\, c\conj{c}^{-1}\det u_\beta^{\,\,\,\alpha} =1
    \right\}
\end{equation*}
and 
\begin{equation*}
    P^\#_+ = \left\{
    \begin{pmatrix}
        1& -v_\beta& d\\
         & \delta_{\beta}^{\,\,\,\alpha}& v^\alpha\\
         & & 1
    \end{pmatrix} \,\,\middle|\,\, \left\lVert\begin{pmatrix}
        d\\ v^\alpha\\ 1
    \end{pmatrix}\right\rVert_h =0
    \right\}.
\end{equation*}
Then $P = G_0 \ltimes P_+$, where $G_0 = G_0^\#/\mathbb{Z}_{n+2}$ and $P_+ = P_+^\#/\mathbb{Z}_{n+2}$.
Note that $P_+^\#$ is the Heisenberg group and is therefore contractible.

The group $P$ admits the following embedding into $\GL(n+2,\mathbb{C})$: 
\begin{equation}
\scalebox{0.9}{$
    P=P^\#/\mathbb{Z}_{n+2} \ni \left[\begin{pmatrix}
        c& -c u_{\beta}^{\,\,\,\gamma}v_\gamma& d\\
         & u_{\beta}^{\,\,\,\alpha}& v^\alpha\\
         & & \conj{c}^{-1}
    \end{pmatrix}\right]
    \mapsto \begin{pmatrix}
        1& -u_{\beta}^{\,\,\,\gamma}v_\gamma& c^{-1}d\\
         & c^{-1} u_{\beta}^{\,\,\,\alpha}& c^{-1}v^\alpha\\
         & & \abs{c}^{-2}
    \end{pmatrix} \in \GL(n+2,\mathbb{C}).
$}
\end{equation}
We denote the image of this embedding by $Q$. 
Under this embedding, $G_0$ is mapped onto the conformal unitary group of matrices of the form 
\begin{equation}
    \begin{pmatrix}
        1& & \\
         & \lambda u_{\beta}^{\,\,\,\alpha}& \\
         & & \lambda^2
    \end{pmatrix}, \quad \lambda>0,\, u_\beta^{\,\,\,\alpha} \in \mathit{U}(n).
\end{equation}
Therefore, the group $Q$ deformation retracts onto its subgroup of matrices of the form 
\begin{equation}
    \begin{pmatrix}
        1& & \\
         & u_{\beta}^{\,\,\,\alpha}& \\
         & & 1
    \end{pmatrix}, \quad  u_\beta^{\,\,\,\alpha} \in \mathit{U}(n).
    \label{eq:unitary_group}
\end{equation}

\subsection{Cartan bundles and connections}
\label{subsec:Cartan_bundles}
A Cartan geometry on a manifold $M$ consists of a principal bundle $\mathcal{G}\longrightarrow M$, called the Cartan bundle, together with a one-form $\omega$ on $\mathcal{G}$, called the Cartan connection. 
This definition generalizes the homogeneous model in which $M=G/P,\, \mathcal{G}=G$, and $\omega$ is the Maurer--Cartan form. 
The curvature $K=d\omega+\frac{1}{2}[\omega,\omega]$ is the complete local obstruction to the Cartan geometry being locally isomorphic to its homogeneous model.

Cartan-geometric descriptions of CR geometry were given independently by Tanaka \cite{tanaka_graded_1966,tanaka_non-degenerate_1976} and Chern \cite{chern_real_1974}.
We follow Chern's construction, with modifications reflecting subsequent developments in parabolic Cartan geometry; see \cite{cap_parabolic_2009,matsumoto_cr_2022}.

Let $M$ be a strictly pseudoconvex CR manifold. 
A choice of a contact form $\theta$ determines the metric $h_\theta$ on $\mathcal{E}(1,0)$.
Let $\zeta$ be a section of $\mathcal{E}(-1,0)^\times$ that is of unit length with respect to $h_\theta^{-1}$.
Then, a local trivialization of $\mathcal{E}(-1,0)^\times$ is given by 
\begin{equation}
    M \times \mathbb{C}^\times \ni (x,c) \mapsto c\zeta_x \in \mathcal{E}(-1,0)^\times.
\end{equation}

Let $\theta,\theta^{\alpha},\theta^{\conj{\alpha}}$ be an admissible coframe, and consider the subbundle $\mathscr{T}^*\subset T^*_{\mathbb{C}}\mathcal{E}(-1,0)^\times$ locally spanned by covectors
\begin{equation}
    -\frac{1}{n+2}\omega_{\gamma}^{\,\,\,\gamma} +\frac{dc}{c} -\frac{\sqrt{-1}}{n+2}P\theta, \theta^\alpha, \sqrt{-1}\abs{c}^2\theta,
    \label{eq:unitary_coframe}
\end{equation}
where 
\begin{equation}
    P_{\alpha \conj{\beta}} = \frac{1}{n+2} \left(R_{\alpha \conj{\beta}} - \frac{1}{2(n+1)}Rl_{\alpha \conj{\beta}}\right), \quad P = P_{\gamma}^{\,\,\,\gamma}.
\end{equation}
Note that the one-form $-\omega_{\gamma}^{\,\,\,\gamma}/(n+2)+dc/c$ is the connection form of the Tanaka--Webster connection induced on $\mathcal{E}(-1,0)$, and is therefore well defined.
We introduce a metric $\bm{h}$ on $\mathscr{T}^*$ whose matrix representation with respect to the frame \eqref{eq:unitary_coframe} is given by 
\begin{equation*}
    \begin{pmatrix}
         & & 1\\
         & \abs{c}^2l_{\alpha\conj{\beta}}& \\
        1& & 
    \end{pmatrix}.
\end{equation*}
If $R_c\,(c\in\mathbb{C}^\times)$ denotes the principal right action on $\mathcal{E}(-1,0)^\times$, then the metric satisfies $R_c^* \bm{h} = \abs{c}^{2} \bm{h}$.

\begin{lem}
    The bundle $\mathscr{T}^*$ and the metric $\bm{h}$ are independent of the choice of $\theta$.
\end{lem}

\begin{proof}
    Let $\widehat{\theta} = e^\Upsilon \theta$ be another contact form, where $\Upsilon$ is a smooth real function.
    An admissible coframe for $\widehat{\theta}$ is obtained by setting $\widehat{\theta}^{\alpha} =\theta^\alpha + \sqrt{-1}\Upsilon^{\alpha}\theta$.
    The corresponding fiber coordinate is given by $\widehat{c} = e^{-\Upsilon/2}c$.
    Then 
    \begin{equation*}
        \frac{d\widehat{c}}{\widehat{c}} = -\frac{1}{2}\left(\Upsilon_\gamma \theta^\gamma + \Upsilon^{\gamma} \theta_{\gamma} + \Upsilon_{T}\theta\right) + \frac{dc}{c},
    \end{equation*}
    where $\Upsilon_{T}=\nabla^{TM}_{T}\Upsilon$.
    The Tanaka--Webster connection associated with $\widehat{\theta}$ satisfies 
    \begin{align}
        \widehat{\omega}_{\gamma}^{\,\,\,\gamma} &= \omega_{\gamma}^{\,\,\,\gamma} + \frac{n+2}{2}(\Upsilon_\gamma \theta^\gamma - \Upsilon^{\gamma} \theta_{\gamma}) \\ &\qquad+ \sqrt{-1} \left(\frac{1}{2} (\Upsilon_{\gamma}^{\,\,\,\gamma} + \Upsilon^{\gamma}_{\,\,\,\gamma}) + (n+1)\Upsilon_{\gamma}\Upsilon^{\gamma}\right)\theta
    \end{align}
    and 
    \begin{equation*}
        \widehat{P} = P - \frac{1}{2} (\Upsilon_{\gamma}^{\,\,\,\gamma} + \Upsilon^{\gamma}_{\,\,\,\gamma}) - \frac{n}{2} \Upsilon_{\gamma}\Upsilon^{\gamma}.
    \end{equation*}
    See \cite{lee_fefferman_1986} for these transformation rules.
    Using these formulas, we obtain
    \begin{equation}
    \begin{aligned}
        &\begin{pmatrix}
            -\frac{1}{n+2} \widehat{\omega}_{\gamma}^{\,\,\,\gamma} +\frac{d\widehat{c}}{\widehat{c}} -\frac{\sqrt{-1}}{n+2} \widehat{P} \widehat{\theta}\\
            \widehat{\theta}^\alpha\\
            \sqrt{-1}\abs{\widehat{c}}^2\widehat{\theta}
        \end{pmatrix}\\
        &=\begin{pmatrix}
            1& -\Upsilon_\beta& -\frac{\abs{c}^{-2}}{2} (\Upsilon_\gamma \Upsilon^\gamma - \sqrt{-1} \Upsilon_{T})\\
             & \delta_{\beta}^{\,\,\,\alpha}&  \abs{c}^{-2}\Upsilon^{\alpha}\\
             & & 1
        \end{pmatrix}
        \begin{pmatrix}
            -\frac{1}{n+2}\omega_{\gamma}^{\,\,\,\gamma} +\frac{dc}{c} -\frac{\sqrt{-1}}{n+2}P\theta\\
            \theta^\alpha\\
            \sqrt{-1}\abs{c}^2\theta
        \end{pmatrix}.
        \label{eq:transformation_formula}
    \end{aligned}
    \end{equation}
    The transformation matrix preserves $\bm{h}$, so the lemma follows.
\end{proof}

The dual bundle $\mathscr{T}$ of $\mathscr{T}^*$ can be realized as a quotient of $T^{\mathbb{C}}\mathcal{E}(-1,0)^\times$.
It inherits a metric from $\bm{h}$, which we again denote by $\bm{h}$. 
Let $\mathcal{G}$ be the bundle of frames $(Z_0,Z_\alpha,Z_{n+1})$ of $\mathscr{T}$ satisfying 
\begin{equation}
\begin{gathered}
    (Z_0,Z_0)_{\bm{h}} = (Z_0,Z_\beta)_{\bm{h}} = (Z_\alpha,Z_{n+1})_{\bm{h}} = (Z_{n+1},Z_{n+1})_{\bm{h}} = 0,\\
    (Z_0,Z_{n+1})_{\bm{h}} = 1, (Z_\alpha,Z_\beta)_{\bm{h}} = \delta_{\alpha \conj{\beta}},\\
    Z_{0} = \left[c\frac{\partial}{\partial c}\right].
\end{gathered}
\end{equation}
Let $\mathcal{T}\to M$ be the bundle of sections of $\mathscr{T}$ that are invariant under the principal right action.
The ambiguity in choosing an $(n+2)$nd root $\mathcal{E}(-1,0)$ of $\mathcal{K}_M$ is given by the action of the group of $(n+2)$nd roots of unity. 
Since the sections defining $\mathcal{T}$ are invariant under this action, $\mathcal{T}$ is globally defined over $M$. 
The bundle $\mathcal{G}$ may then be regarded as a frame bundle of $\mathcal{T}$;
it is a principal $Q$-bundle over $M$. 
This principal bundle is called the \emph{CR Cartan bundle} of $M$.

In \cite{chern_real_1974}, Chern constructed a family of one-forms on $\mathcal{G}$ satisfying a structure equation with respect to the coframe \eqref{eq:unitary_coframe}.
These forms do not define a connection on $\mathcal{T}$, since they are not equivariant under the action of $Q$.
Instead, they define a connection on $\mathcal{T}^\#=\mathcal{E}(-1,0)\otimes\mathcal{T}$.
The bundle $\mathcal{T}^\#$ is called the \emph{standard tractor bundle}.
The induced connection is called the \emph{normal tractor connection} and is denoted by $\nabla^{\mathcal{T}^\#}$.

The bundle $\mathcal{T}^\#$ is identified with the bundle of sections of $\mathscr{T}$ that are homogeneous of degree $(-1,0)$, that is, sections $\bm{v}$ satisfying $(R_c)_*\bm{v}=c\bm{v}$.
To a frame $(Z_A)\in\mathcal{G}$ of $\mathcal{T}$, one associates a frame of $\mathcal{T}^\#$ as follows.
Let $(\theta^A)$ be the dual coframe.
The covectors $\theta^1,\ldots,\theta^{n+1}$ are basic, since they annihilate the infinitesimal generator $Z_0$ of the principal $\mathbb{C}^\times$-action.
Let $\zeta$ be a section of $\mathcal{E}(-1,0)^\times$ satisfying
\begin{equation}
    \zeta^{n+2}=\theta^1\wedge\cdots\wedge\theta^{n+1},
\end{equation}
and let $c$ be the corresponding fiber coordinate. 
Then $(c^{-1}Z_A)$ is a frame of $\mathcal{T}^\#$.
The bundle $\mathcal{G}^\#$ of such frames of $\mathcal{T}^\#$ is a principal $P^\#$-bundle over $M$. 
It is also an $(n+2)$-fold covering of $\mathcal{G}$, corresponding to the $n+2$ possible choices of $\zeta$. 
From this point of view, $\mathcal{G}$ may be regarded as a principal $P$-bundle: a point of $\mathcal{G}$ represents the collection of the $n+2$ frames of $\mathcal{T}^\#$ arising from the possible choices of $\zeta$.

\begin{rem}
    A choice of a contact form $\theta$ determines a frame of $\mathscr{T}^*$ of the form \eqref{eq:unitary_coframe}. 
    We may regard it as a weighted frame of $(\mathcal{T}^\#)^*$ by setting 
    \begin{align}
        Y_A &= -\frac{1}{n+2}\omega_{\gamma}^{\,\,\,\gamma} +\frac{dc}{c} -\frac{\sqrt{-1}}{n+2}P\theta \in (\mathcal{T}^\#)^*\otimes\mathcal{E}(-1,0),\\
        W_A^{\,\,\,\alpha} &= \theta^\alpha \in (\mathcal{T}^\#)^*\otimes\mathcal{E}(-1,0),\\
        Z_A &= \sqrt{-1}\abs{c}^2\theta \in (\mathcal{T}^\#)^*\otimes\mathcal{E}(0,1).
    \end{align}
    Under a change $\widehat{\theta} = e^{\Upsilon}\theta$ of contact form, the same transformation formula as in \eqref{eq:transformation_formula} holds for the weighted frame $(Y_A,W_A^{\,\,\,\alpha},Z_A)$.
    This recovers the definition of the CR standard co-tractor bundle in \cite{gover_cr_2005}.
\end{rem}

\subsection{Exact Weyl structures}
\label{subsec:Exact}
As explained in Subsection \ref{subsec:Groups}, the structure group $Q$ of $\mathcal{G}$ deformation retracts onto the unitary group of matrices of the form \eqref{eq:unitary_group}, which we also denote by $\mathit{U}(n)$.
A choice of a contact form $\theta$ determines a reduction of the structure group to $\mathit{U}(n)$, which we explain below. 
This reduction can be extended to the group $G_0$, and such an extension is referred to as an \emph{exact Weyl structure} \cite{cap_parabolic_2009,matsumoto_cr_2022}.

Once a contact form $\theta$ is fixed, the unitary admissible coframes $\theta,\theta^\alpha,\theta^{\conj{\alpha}}$ are parametrized by the action of $\mathit{U}(n)$.
We denote the bundle of corresponding unitary frames $(Z_\alpha)$ of $T^{1,0}M$ by $\mathcal{F}_{\theta}(T^{1,0}M)$.
Taking the frame dual to  
\begin{equation}
    -\frac{1}{n+2}\omega_{\gamma}^{\,\,\,\gamma} +\frac{dc}{c} -\frac{\sqrt{-1}}{n+2}P\theta, \theta^\alpha, \sqrt{-1}\theta,
\end{equation}
we obtain a reduction $\sigma_\theta \colon \mathcal{F}_{\theta}(T^{1,0}M) \to \mathcal{G}$ of structure group to $\mathit{U}(n)$.
As a corollary, the bundle-theoretic classification of the Cartan bundle $\mathcal{G}$ is equivalent to that of the CR holomorphic tangent bundle $T^{1,0}M$.
In particular, $\mathcal{G}$ is trivial if and only if $T^{1,0}M$ is trivial.

Now let $\bm{Z}=(Z_\alpha)$ be a section of $\mathcal{F}_{\theta}(T^{1,0}M)$. 
In \cite{matsumoto_cr_2022}, Matsumoto computed the connection form of the normal tractor connection $\nabla^{\mathcal{T}^\#}$ with respect to any of the $n+2$ frames of $\mathcal{T}^\#$ represented by $\sigma_\theta(\bm{Z})$.
The result is 
\begin{equation}
\scalebox{0.8}{$
    \begin{pmatrix}
        -\frac{1}{n+2}\omega_{\gamma}^{\,\,\,\gamma}-\frac{\sqrt{-1}}{n+2}P\theta& -\sqrt{-1}A_{\beta\rho}\theta^{\rho}-P_{\beta\conj{\sigma}}\theta^{\conj{\sigma}}-2\sqrt{-1}T_{\beta}\theta& T_{\rho}\theta^{\rho}+T_{\conj{\sigma}}\theta^{\conj{\sigma}}-\sqrt{-1}S\theta\\
        \theta^\alpha& \omega_{\beta}^{\,\,\,\alpha}-\frac{1}{n+2}\omega_{\gamma}^{\,\,\,\gamma}\delta_{\beta}^{\,\,\,\alpha}+\sqrt{-1}\tf P_{\beta}^{\,\,\,\alpha}\theta& P_{\,\,\,\rho}^{\alpha}\theta^{\rho}-\sqrt{-1}A_{\,\,\,\conj{\sigma}}^{\alpha}\theta^{\conj{\sigma}}-2\sqrt{-1}T^{\alpha}\theta\\
        \sqrt{-1}\theta& -\theta_\beta&  -\frac{1}{n+2}\omega_{\gamma}^{\,\,\,\gamma}-\frac{\sqrt{-1}}{n+2}P\theta
    \end{pmatrix},
    \label{eq:exact_Weyl}
$}
\end{equation}
where 
\begin{align}
    \tf P_{\beta}^{\,\,\,\alpha} &= P_{\beta}^{\,\,\,\alpha} - \frac{1}{n+2}P\delta_{\beta}^{\,\,\,\alpha}\\
    T_\alpha &= \frac{1}{n+2}(P_{,\alpha}-\sqrt{-1}A_{\alpha\gamma,}^{\quad\gamma}),\\
    S &= -\frac{1}{n}(T_{\gamma,}^{\,\,\,\gamma}+T_{\conj{\gamma},}^{\,\,\,\conj{\gamma}}+P_{\alpha\conj{\beta}}P^{\alpha\conj{\beta}}-A_{\alpha\beta}A^{\alpha\beta}).
\end{align}

\subsection{Pseudo-Einstein structures}
Let $M$ be a strictly pseudoconvex CR manifold and fix a contact form $\theta$.
As explained in Subsection \ref{subsec:Basic}, the contact form $\theta$ determines a metric $h_\theta \in \mathcal{E}_{\mathbb{R}}(-1,-1)$ on $\mathcal{E}(1,0)$.
Define a metric connection $D^\theta$ on $\mathcal{E}(1,0)$ by 
\begin{equation}
    D^\theta \zeta = \nabla^{TM} \zeta -\frac{\sqrt{-1}}{n(n+2)}R\theta \otimes \zeta.
    \label{eq:TanakaConn}
\end{equation}

\begin{thm}[\cite{hirachi_variation_2017}]
\label{thm:pseudo_Einstein}
The following conditions on $\theta$ are equivalent: 
\begin{enumerate}
    \item The connection $D^\theta$ is flat.
    \item Around each point $x\in M$, there exists a closed local section of $\mathcal{K}_M$ having unit length with respect to $h_\theta^{-n-2}$.
    \item The Tanaka--Webster connection satisfies \begin{equation}
        \begin{cases}
            R_{\alpha \conj{\beta}} = \frac{1}{n}Rl_{\alpha \conj{\beta}} & n>1, \\
            \sqrt{-1}A_{\alpha \gamma,}^{\quad\gamma} = \frac{1}{n}R_{,\alpha} & n=1. 
        \end{cases}
        \label{eq:pseudo_Einstein}
    \end{equation}
\end{enumerate}
\end{thm}

If a contact form $\theta$ satisfies one of the above equivalent conditions, it is called a \emph{pseudo-Einstein contact form} \cite{lee_pseudo-einstein_1988}.
In that case, both formulas in \eqref{eq:pseudo_Einstein} are valid in every dimension.
For example, the boundary of a bounded strictly pseudoconvex domain in $\mathbb{C}^{n+1}$ admits a pseudo-Einstein contact form.

\section{Secondary Invariants of the Cartan Connection}
In this section, we construct global $\mathbb{R}/\mathbb{Z}$-valued CR invariants by applying Cheeger--Simons theory \cite{chern_characteristic_1974,cheeger_differential_1985} to a globally defined modification of the normal tractor connection.
This construction unifies and extends the Burns--Epstein invariants \cite{burns_global_1988,burns_characteristic_1990} and Marugame's invariant \cite{marugame_renormalized_2016} by realizing them as $\mathbb{R}$-valued lifts of differential characters.

\subsection{Chern--Simons forms}
Let $E\to M$ be a complex vector bundle of rank $k$ over a manifold $M$.
Let $\Phi \in I^l(\mathfrak{gl}(k,\mathbb{C}))$ be an $\Ad$-invariant polynomial of degree $l$.
For a connection $\nabla$ on $E$, we denote by $\Phi(\nabla)$ the corresponding characteristic form.

Suppose we have two connections $\nabla^{(0)},\nabla^{(1)}$ on $E$.
Let $\omega^{(t)}$ and $\Omega^{(t)}$ denote the connection and curvature forms of $\nabla^{(t)} = (1-t)\nabla^{(0)} + t\nabla^{(1)}$, respectively, with respect to a fixed local frame.
Denote by $\widetilde{\Phi}\colon \mathfrak{gl}(k,\mathbb{C})^{\otimes l} \to \mathbb{C}$ the polarization of $\Phi$, normalized by 
\begin{equation}
    \widetilde{\Phi}(A,\ldots,A) = \Phi(A).
\end{equation}
The \emph{relative Chern--Simons form} is then defined by  
\begin{equation}
    \Phi(\nabla^{(0)},\nabla^{(1)}) = l\int_0^1 \widetilde{\Phi}(\omega^{(1)}-\omega^{(0)},\Omega^{(t)},\ldots,\Omega^{(t)}) dt.
    \label{eq:relativeCS}
\end{equation}
The defining property of the relative Chern--Simons form is that it satisfies 
\begin{equation}
    d\Phi(\nabla^{(0)},\nabla^{(1)}) = \Phi(\nabla^{(1)})-\Phi(\nabla^{(0)}).
\end{equation}
From \eqref{eq:relativeCS}, one computes 
\begin{equation}
\label{eq:productCS}
\scalebox{0.95}{$
\begin{aligned}
    &(\Phi\cdot\Psi)(\nabla^{(0)},\nabla^{(1)})-\Phi(\nabla^{(0)})\wedge\Psi(\nabla^{(0)},\nabla^{(1)})-\Phi(\nabla^{(0)},\nabla^{(1)})\wedge\Psi(\nabla^{(1)})\\ &= lm\cdot d\int_{0}^{1}dt\int_{0}^{t}ds \widetilde{\Phi}(\omega^{(1)}-\omega^{(0)},\Omega^{(s)},\ldots,\Omega^{(s)})\wedge\widetilde{\Psi}(\omega^{(1)}-\omega^{(0)},\Omega^{(t)},\ldots,\Omega^{(t)})
\end{aligned}
$}
\end{equation}
for invariant polynomials $\Phi$ and $\Psi$ of degree $l$ and $m$, respectively.
Moreover, for another connection $\nabla^{(2)}$ on $E$, one has 
\begin{equation}
\label{eq:difference_form}
    \Phi(\nabla^{(0)},\nabla^{(1)}) + \Phi(\nabla^{(1)},\nabla^{(2)}) = \Phi(\nabla^{(0)},\nabla^{(2)}) + \text{exact form}.
\end{equation}

Let $\pi \colon \mathcal{F}(E) \to M$ be the full frame bundle of $E$, and denote by $\bm{s}_{\text{taut}}$ its tautological frame.
Let $\nabla^{\bm{s}_{\text{taut}}}$ be the trivial connection on $\pi^*E$ with respect to the tautological frame.
For a connection $\nabla$ on $E$, the relative Chern--Simons form 
\begin{equation}
    T_\Phi(\nabla) = \Phi(\nabla^{\bm{s}_{\text{taut}}},\pi^*\nabla)
\end{equation}
is called the \emph{Chern--Simons form}.
This is a $(2l-1)$-form on $\mathcal{F}(E)$ satisfying 
\begin{equation}
    dT_\Phi(\nabla) = \pi^*\Phi(\nabla).
\end{equation}

\subsection{Cheeger--Simons differential characters}
\label{subsec:Cheeger--Simons_differential}
Let $\pi\colon \mathrm{St}(r,E) \to M$ be the Stiefel bundle of $r$-frames of $E$.
A point of $\mathrm{St}(r,E)$ is an ordered $r$-tuple of linearly independent vectors in a fiber of $E$.
Denote by $\bm{s}_{\text{taut}}$ the tautological $r$-frame of $\pi^*E$.
Let $\nabla^{\bm{s}_{\text{taut}}}$ be a connection on $\pi^*E$ for which $\bm{s}_{\text{taut}}$ is parallel.
For a connection $\nabla$ on $E$, the form
\begin{equation}
    c_q(\nabla^{\bm{s}_{\text{taut}}},\pi^*\nabla) \quad (q = k-r+1)
\end{equation}
is independent modulo exact forms of the choice of $\nabla^{\bm{s}_{\text{taut}}}$.
To see this, let $\widetilde{\nabla}^{\bm{s}_{\text{taut}}}$ be another such connection.
Using \eqref{eq:difference_form}, we have 
\begin{equation}
    c_q(\nabla^{\bm{s}_{\text{taut}}},\pi^*\nabla) - c_q(\widetilde{\nabla}^{\bm{s}_{\text{taut}}},\pi^*\nabla) = c_q(\nabla^{\bm{s}_{\text{taut}}},\widetilde{\nabla}^{\bm{s}_{\text{taut}}}) + \text{exact form}.
\end{equation}
Now $c_q(\nabla^{\bm{s}_{\text{taut}}},\widetilde{\nabla}^{\bm{s}_{\text{taut}}}) = 0$.
Indeed, every connection in the affine family $(1-t)\nabla^{\bm{s}_{\text{taut}}} + t\widetilde{\nabla}^{\bm{s}_{\text{taut}}}$ makes $\bm{s}_{\text{taut}}$ parallel, and hence the difference form vanishes.
This proves the claim.

Although this form need not be closed on $\mathrm{St}(r,E)$, its restriction to each fiber is closed because $\pi^*c_q(\nabla)$ vanishes on the fibers.
Letting $x_{2q-1}$ denote the canonical generator of $H_{2q-1}(\mathrm{St}(r,\mathbb{C}^k))$, one computes 
\begin{equation}
    \int_{x_{2q-1}} c_q(\nabla^{\bm{s}_{\text{taut}}},\pi^*\nabla) = 1.
\end{equation}

Recall the exact sequence 
\begin{equation}
    \mathbb{Z} \simeq H_{2q-1}(\mathrm{St}(r,\mathbb{C}^k),\mathbb{Z}) \to H_{2q-1}(\mathrm{St}(r,E),\mathbb{Z}) \stackrel{\pi_*}{\longrightarrow} H_{2q-1}(M,\mathbb{Z}) \to 0.
    \label{eq:exact_sequence}
\end{equation}
For any smooth cycle $z \in Z_{2q-1}(M,\mathbb{Z})$, there exist a smooth cycle $y \in Z_{2q-1}(\mathrm{St}(r,E),\mathbb{Z})$ and a smooth chain $w \in C_{2q}(M,\mathbb{Z})$ such that $z = \pi_*y + \partial w$.
The integral 
\begin{equation}
    S_{c_q}(\nabla)(z) = \int_y c_q(\nabla^{\bm{s}_{\text{taut}}},\pi^*\nabla) + \int_w c_q(\nabla) + \mathbb{Z} \in \mathbb{R}/\mathbb{Z}
\end{equation}
is independent of the choices of $y$ and $w$.
The reduction modulo $\mathbb{Z}$ is necessary, essentially because the form $c_q(\nabla^{\bm{s}_{\text{taut}}},\pi^*\nabla)$ has nontrivial periods along cycles in each fiber.
The resulting homomorphism 
\begin{equation}
    S_{c_q}(\nabla) \colon Z_{2q-1}(M,\mathbb{Z}) \to \mathbb{R}/\mathbb{Z}
\end{equation}
is called the \emph{Cheeger--Simons differential character} \cite{cheeger_differential_1985} associated with the Chern polynomial $c_q$.

If $E$ admits a global $r$-frame $\bm{s}$, $S_{c_q}(\nabla)$ admits an $\mathbb{R}$-valued lift by setting 
\begin{equation}
    \widetilde{S}_{c_q}^{\bm{s}}(\nabla)(z) \coloneqq \int_{\bm{s}_*z} c_q(\nabla^{\bm{s}_{\text{taut}}},\pi^*\nabla) = \int_z c_q(\nabla^{\bm{s}},\nabla),
\end{equation}
where $\nabla^{\bm{s}}$ denotes a connection on $E$ for which $\bm{s}$ is parallel.
Given another global $r$-frame $\bm{t}$, we have, by \eqref{eq:difference_form},
\begin{equation}
    \widetilde{S}_{c_q}^{\bm{s}}(\nabla)(z) - \widetilde{S}_{c_q}^{\bm{t}}(\nabla)(z) = \int_z c_q(\nabla^{\bm{s}},\nabla^{\bm{t}}).
\end{equation}
The right-hand side can also be expressed as an evaluation of the difference cocycle $d(\bm{s},\bm{t})$ from obstruction theory.
This follows from the fact that $c_q(\nabla^{(0)},\nabla^{(1)})$ is independent, modulo exact forms, of the choice of homotopy $\nabla^{(t)}$.
As a consequence, if $\bm{s}$ and $\bm{t}$ are homotopic through $r$-frames, then we have $\widetilde{S}_{c_q}^{\bm{s}}(\nabla) = \widetilde{S}_{c_q}^{\bm{t}}(\nabla)$.

The differential characters $S_{c_{q_i}}(\nabla)\,(i=1,2,\ldots,m)$ can be multiplied to obtain a differential character 
\begin{equation}
    S_{c_{q_1}}(\nabla)*\cdots*S_{c_{q_m}}(\nabla) \colon Z_{2(q_1+\cdots+q_m)-1}(M,\mathbb{Z}) \to \mathbb{R}/\mathbb{Z}.
\end{equation}
We do not recall the general definition of this product here; see \cite{cheeger_differential_1985}.
Instead, we describe its $\mathbb{R}$-valued lifts in terms of differential forms.

Suppose, for example, that $E$ admits an $r_1=(k-q_1+1)$-frame $\bm{s}_1$.
We then define 
\begin{equation}
    \left(\widetilde{S}_{c_{q_1}}^{\bm{s}_1}(\nabla)*S_{c_{q_2}}(\nabla)*\cdots*S_{c_{q_m}}(\nabla)\right)(z) = \int_z c_{q_1}(\nabla^{\bm{s}_1},\nabla) \wedge (c_{q_2}\cdots c_{q_m})(\nabla).
\end{equation}
If $m\ge2$ and $E$ also admits an $r_2=(k-q_2+1)$-frame $\bm{s}_2$, then, since $z$ is a cycle, 
\begin{equation}
\scalebox{0.95}{$
\begin{aligned}
    0 &= -\int_z d\left(c_{q_1}(\nabla^{\bm{s}_1},\nabla) \wedge c_{q_2}(\nabla^{\bm{s}_2},\nabla) \wedge (c_{q_3}\cdots c_{q_m})(\nabla)\right)\\
    &= \int_z c_{q_1}(\nabla^{\bm{s}_1},\nabla) \wedge (c_{q_2}\cdots c_{q_m})(\nabla) - \int_z c_{q_1}(\nabla) \wedge c_{q_2}(\nabla^{\bm{s}_2},\nabla) \wedge (c_{q_3}\cdots c_{q_m})(\nabla).
\end{aligned}
    $}
\end{equation}
Consequently, if at least two of the factors admit $\mathbb{R}$-valued lifts, the resulting lift of the product is independent of the partial frames used in its definition.

\subsection{Secondary invariants}
\label{subsec:Secondary}
Let $\mathcal{T}^\#$ and $\nabla^{\mathcal{T}^\#}$ denote the standard tractor bundle and the normal tractor connection, respectively; see Subsection \ref{subsec:Cartan_bundles}.
We would like to consider the Cheeger--Simons differential characters associated with $\nabla^{\mathcal{T}^\#}$.
However, since $\mathcal{E}(-1,0)$, and hence $\mathcal{T}^\# = \mathcal{E}(-1,0)\otimes\mathcal{T}$, may exist only locally, we instead seek a globally defined connection that retains the curvature information carried by $\nabla^{\mathcal{T}^\#}$.

Let $M$ be a compact strictly pseudoconvex CR manifold admitting a pseudo-Einstein contact form $\theta$, and let $D^\theta$ be the flat connection on $\mathcal{E}(1,0)$ determined by $\theta$.
We define $\nabla^{\mathcal{T}} = D^\theta \otimes \nabla^{\mathcal{T}^\#}$, which is a globally defined connection on $\mathcal{T}=\mathcal{E}(1,0)\otimes\mathcal{T}^\#$.

\begin{dfn}
    For positive integers $q_1,\ldots,q_m$ satisfying $q_1+\cdots+q_m=n+1$, we define 
    \begin{equation}
        \mu_{c_{q_1}\cdots c_{q_m}}(M) \coloneqq \left(S_{c_{q_1}}(\nabla^{\mathcal{T}})*\cdots*S_{c_{q_m}}(\nabla^{\mathcal{T}})\right)(M) \in \mathbb{R}/\mathbb{Z}.
    \end{equation}
    By Theorem \ref{thm:CR_invariance} below, this value is independent of the choice of pseudo-Einstein contact form $\theta$.
    We call it the \emph{CR Cheeger--Simons invariant}.
\end{dfn}

In what follows, we present two important situations in which the CR Cheeger--Simons invariant admits an $\mathbb{R}$-valued lift.

\begin{enumerate}
    \item Suppose first that $\mathcal{T}$ is trivial as a vector bundle.
    This is the case, for instance, when $M$ is embedded in $\mathbb{C}^{n+1}$.
    Since $\mathcal{T}$ admits an $r$-frame for every $1\le r\le n+2$, the discussion in Subsection \ref{subsec:Cheeger--Simons_differential} shows that, for $m\ge2$, the value 
    \begin{equation}
        \widetilde{\mu}_{c_{q_1}\cdots c_{q_m}}(M) \coloneqq \int_M c_{q_1}(\nabla^{\bm{s}_1},\nabla^{\mathcal{T}}) \wedge (c_{q_2}\cdots c_{q_m})(\nabla^{\mathcal{T}})
    \end{equation}
    is a lift of $\mu_{c_{q_1}\cdots c_{q_m}}(M)$ and is independent of the chosen partial frames.
    In contrast, when $m=1$, the integral 
    \begin{equation}
        \int_M c_{n+1}(\nabla^{\bm{s}_1},\nabla^{\mathcal{T}})
    \end{equation}
    generally depends on the choice of the $2$-frame $\bm{s}_1$.

    \item Independently of the triviality of $\mathcal{T}$, there is a canonical choice of such a $2$-frame in the case $m=1$.
    Let $\sigma_\theta \colon \mathcal{F}_{\theta}(T^{1,0}M) \to \mathcal{G}$ be the reduction of structure group to $\mathit{U}(n)$ described in Subsection \ref{subsec:Exact}.
    Among the frames of $\mathcal{T}$ corresponding to $\sigma_\theta(\mathcal{F}_{\theta}(T^{1,0}M))$, the vectors $Z_0$ and $Z_{n+1}$ are invariant under the action of $\mathit{U}(n)$.
    We set $\bm{s}_1=(Z_0,Z_{n+1})$ and define
    \begin{equation}
        \widetilde{\mu}_{c_{n+1}}(M) \coloneqq \int_M c_{n+1}(\nabla^{\bm{s}_1},\nabla^{\mathcal{T}}).
    \end{equation}
    The contact form used to specify the exact Weyl structure need not agree with the pseudo-Einstein contact form used to define $\nabla^{\mathcal{T}}$, because any two choices of the former determine homotopic $2$-frames $(Z_0,Z_{n+1})$.
\end{enumerate}

Since the curvature of the normal tractor connection $\nabla^{\mathcal{T}^\#}$ is trace-free and $D^\theta$ is flat, we have $c_1(\nabla^{\mathcal{T}})=0$.
In particular, we obtain the following result. 
\begin{prop}
    Let $\Phi$ be an invariant polynomial of the form $\Phi = c_{q_1}\cdots c_{q_m}$, where $q_1 \ge q_2 \ge \cdots \ge q_m$ and $q_1 + \cdots + q_m = n$. 
    Assume that the bundle $\mathcal{T}$ admits an $r_1=((n+2)-q_1+1)$-frame.
    Then 
    \begin{equation}
        \mu_{\Phi \cdot c_1}(M) = 0.
    \end{equation}
    
    For the corresponding $\mathbb{R}$-valued lift, assume in addition that $\mathcal{T}$ admits an $(n+3-q_2)$-frame when $m\ge2$, and that $\mathcal{T}$ is trivial when $m=1$. 
    Then
    \begin{equation}
        \widetilde{\mu}_{\Phi\cdot c_1}(M) = 0.
    \end{equation}
    
    Consequently, if there exists such a Chern monomial $\Phi$ for which $\mu_{\Phi \cdot c_1}(M) \ne 0$, then $\mathcal{T}$ admits no $r_1=((n+2)-q_1+1)$-frame.
    In particular, $M$ cannot be CR embedded in $\mathbb{C}^{n+1}$.
    \label{thm:obstruction_to_embedding}
\end{prop}

We now state the main theorem of this section.

\begin{thm}
    The CR Cheeger--Simons invariant $\mu_{c_{q_1}\cdots c_{q_m}}(M)$ is independent of the choice of pseudo-Einstein contact form and hence is a CR invariant. 
    Each of the $\mathbb{R}$-valued lifts described above, whenever defined, is likewise independent of this choice.
    \label{thm:CR_invariance}
\end{thm}
The proof is postponed until the next section, where we establish the reduction and ambient constructions needed for it.

\subsection{When the CR holomorphic tangent bundle is trivial}
Originally, the Burns--Epstein invariant was defined for a three-dimensional compact strictly pseudoconvex CR manifold $M$ whose CR holomorphic tangent bundle $T^{1,0}M$ is trivial (as a vector bundle) \cite{burns_global_1988}. 
Then the Cartan bundle $\mathcal{G}$ admits a global section, via which the Chern--Simons form of the normal Cartan connection can be pulled back to $M$. 
The integral of the Chern--Simons form over $M$ gives the Burns--Epstein invariant. 
We present here a topological argument showing that the integral is independent of the choice of section of $\mathcal{G}$, thereby generalizing the original Burns--Epstein invariant to arbitrary dimensions. 
Then we prove that the generalized invariant coincides with our secondary invariant defined above when $M$ admits a pseudo-Einstein structure. 

\begin{rem}
    The Lee conjecture \cite{lee_pseudo-einstein_1988} states that a compact strictly pseudoconvex CR manifold admits a pseudo-Einstein structure if and only if $c_{1}(T^{1,0}M)=0$. 
    If true, then $T^{1,0}M$ being trivial ensures that $M$ admits a pseudo-Einstein structure.
\end{rem}

Let $M$ be a compact strictly pseudoconvex CR manifold of dimension $2n+1$, and assume that $T^{1,0}M$ is trivial.
Then the Cartan bundle $\pi \colon \mathcal{G} \to M$ is trivial (see Subsection \ref{subsec:Exact}), and so is $\pi^\# \colon \mathcal{G}^\# \to M$.
Let $\bm{s}^\#$ be a global section of $\mathcal{G}^\#$, which we identify with a frame of the standard tractor bundle $\mathcal{T}^\#$ (see Subsection \ref{subsec:Cartan_bundles}).
We define the generalized \emph{Burns--Epstein invariant} by 
\begin{equation}
    \widetilde{\mu}_{\Phi}^{\rm BE}(M) \coloneqq \int_{M} (\bm{s}^\#)^* T_\Phi (\nabla^{\mathcal{T}^\#}) = \int_M \Phi(\nabla^{\bm{s}^\#},\nabla^{\mathcal{T}^\#}),
    \label{eq:generalized_Burns-Epstein}
\end{equation}
where $\Phi = c_{q_1}\cdots c_{q_m} \, (q_1+\cdots +q_m = n+1)$.
When $n=1$ and $\Phi=c_2$, this coincides with the original Burns--Epstein invariant.

\begin{thm}
    The integral \eqref{eq:generalized_Burns-Epstein} is independent of the choice of $\bm{s}^\#$.
\end{thm}

\begin{proof}
    Let $\bm{t}^\#$ be another section of $\mathcal{G}^\#$.
    Since $\mathcal{G}^\#$ is trivial, there exists a smooth chain $u \in C_{2n+2}(\mathcal{G}^\#,\mathbb{Z})$ and a smooth cycle $x \in Z_{2n+1}(P^\#,\mathbb{Z})$ such that 
    \begin{equation}
        \bm{t}^\#_*M - \bm{s}^\#_*M = \partial u + i_* x,
    \end{equation}
    where $i \colon P^\# \hookrightarrow \mathcal{G}^\#$ is the inclusion into a fixed fiber.
    We have 
    \begin{equation}
        \int_{\partial u} T_\Phi (\nabla^{\mathcal{T}^\#}) = \int_u (\pi^\#)^*\Phi(\nabla^{\mathcal{T}^\#}) = \int_{\pi^\#_*u} \Phi(\nabla^{\mathcal{T}^\#}) =0
    \end{equation}
    for dimensional reasons.
    On the other hand, the pullback $i^*T_\Phi(\nabla^{\mathcal{T}^\#})$ equals $T_\Phi(\omega_{\rm MC})$, where we regard the Maurer--Cartan form $\omega_{\rm MC}$ of $P^\#$ as the unique connection on the bundle $P^\# \to \{*\}$.
    It then suffices to show 
    \begin{equation}
        [T_\Phi(\omega_{\rm MC})] = 0 \in H^{2n+1}(P^\#,\mathbb{R}).
    \end{equation}

    We start with a description of the cohomology ring $H^\bullet(P^\#,\mathbb{Z})$.
    $P^\#$ deformation retracts onto its subgroup 
    \begin{equation}
        H^\# = \left\{\begin{pmatrix}
            c& & \\
             & u_{\beta}^{\,\,\,\alpha}& \\
             & & c
        \end{pmatrix} \,\middle|\, \begin{gathered}
            c \in \mathit{U}(1), u_{\beta}^{\,\,\,\alpha} \in \mathit{U}(n),\\
            c^2 \det u_{\beta}^{\,\,\,\alpha} = 1
        \end{gathered} \right\}.
    \end{equation}
    $H^\#$ is homeomorphic to $\mathit{U}(1) \times \SU(n)$ via 
    \begin{equation}
        \begin{pmatrix}
            c& & \\
             & u_{\beta}^{\,\,\,\alpha}& \\
             & & c
        \end{pmatrix} \mapsto (c,{\rm diag}(c^2,1,\ldots,1)u_{\beta}^{\,\,\,\alpha}).
    \end{equation}
    The basic facts are that 
    \begin{equation}
        H^\bullet(\mathit{U}(1),\mathbb{Z}) = \wedge_{\mathbb{Z}}^\bullet (x^1), \quad H^\bullet(\SU(n),\mathbb{Z}) = \wedge_{\mathbb{Z}}^\bullet (x^3,\ldots,x^{2n-1}),
    \end{equation}
    where the superscript of $x$ indicates its degree.
    Therefore, 
    \begin{equation}
        H^\bullet(P^\#,\mathbb{Z}) = \wedge_{\mathbb{Z}}^\bullet(x^1,x^3,\ldots,x^{2n-1}).
    \end{equation}

    Let $\mu\colon P^{\#}\times P^{\#}\to P^{\#}$ denote the multiplication. 
    An element $x\in H^{\bullet}(P^{\#},\mathbb{Z})$ of the cohomology ring is said to be primitive if $\mu^{*}x=x\otimes1+1\otimes x$. 
    The subgroup of primitive elements is denoted by $\mathcal{P}H^{\bullet}(P^{\#},\mathbb{Z})$. 
    By the Hopf--Borel theorem for Lie groups \cite{borel_sur_1953}, the generators $x^{1},x^{3},\ldots,x^{2n-1}$ can be taken to be primitive. 
    Hence the primitive subgroup $\mathcal{P}H^{\bullet}(P^{\#},\mathbb{Z})$ is generated by $x^{1},x^{3},\ldots,x^{2n-1}$, and therefore contains no nonzero element of degree $2n+1$. 
    On the other hand, the class represented by $T_\Phi(\omega_{\rm MC})$ is primitive by Borel's characterization of universally transgressive classes \cite{borel_sur_1953}.
    It must therefore vanish.
\end{proof}

\begin{thm}
    If $M$ admits a pseudo-Einstein contact form $\theta$ and $T^{1,0}M$ is trivial, then \begin{equation}
        \widetilde{\mu}_{\Phi}^{\rm BE}(M) = \widetilde{\mu}_{\Phi}(M).
    \end{equation}
    \label{thm:correspondence_with_classicalBE}
\end{thm}

\begin{proof}
    Let $\bm{s}^\#$ be a global section of $\mathcal{G}^\#$.
    Projecting $\bm{s}^\#$ along the $(n+2)$-fold cover $\mathcal{G}^\# \to \mathcal{G}$, we obtain a global frame $\bm{s}=(Z_A)$ of $\mathcal{T}$.
    We have
    \begin{equation}
        \widetilde{\mu}_\Phi(M) = \int_M \Phi(\nabla^{\bm{s}},D^\theta \otimes \nabla^{\mathcal{T}^\#}).
    \end{equation}

    By the cocycle identity for relative Chern--Simons forms, 
    \begin{equation}
        \Phi(\nabla^{\bm{s}},D^\theta \otimes \nabla^{\mathcal{T}^\#}) = \Phi(\nabla^{\bm{s}},D^\theta \otimes \nabla^{\bm{s}^\#}) + \Phi(D^\theta \otimes \nabla^{\bm{s}^\#},D^\theta \otimes \nabla^{\mathcal{T}^\#})
    \end{equation}
    modulo an exact form.

    As a frame of $\mathcal{T}^\#$, $\bm{s}^\#$ can be written as $(c^{-1}Z_A)=(\zeta\otimes Z_A)$, where $\zeta$ is the nowhere-vanishing section of $\mathcal{E}(-1,0)^\times$ determined by $(Z_A)$, and $c$ is the corresponding fiber coordinate (see Subsection \ref{subsec:Cartan_bundles}).
    Let $\eta$ denote the connection form of $D^\theta$ with respect to the frame $\zeta^{-1}$.
    With respect to $\bm{s}$, the connection form of $(1-t)\nabla^{\bm{s}} + tD^\theta \otimes \nabla^{\bm{s}^\#}$ is $t\eta \delta_{B}^{\,\,\,A}$.
    Since $D^\theta$ is flat, this family has vanishing curvature.
    Hence $\Phi(\nabla^{\bm{s}},D^\theta \otimes \nabla^{\bm{s}^\#}) = 0$ by  \eqref{eq:relativeCS}.

    On the other hand, the difference between the connection forms of $D^\theta \otimes \nabla^{\bm{s}^\#}$ and $D^\theta \otimes \nabla^{\mathcal{T}^\#}$ is the same as that of $\nabla^{\bm{s}^\#}$ and $\nabla^{\mathcal{T}^\#}$, and the curvature form of $(1-t)D^\theta \otimes \nabla^{\bm{s}^\#} + tD^\theta \otimes \nabla^{\mathcal{T}^\#}$ is the same as that of $(1-t)\nabla^{\bm{s}^\#} + t\nabla^{\mathcal{T}^\#}$. 
    Therefore, 
    \begin{equation}
        \Phi(D^\theta \otimes \nabla^{\bm{s}^\#},D^\theta \otimes \nabla^{\mathcal{T}^\#}) = \Phi(\nabla^{\bm{s}^\#},\nabla^{\mathcal{T}^\#}),
    \end{equation}
    which proves the theorem.
\end{proof}

\section{Reduction and Ambient Construction}
In this section, we introduce a corank-one subbundle $\underline{\mathcal{T}} \subset \mathcal{T}$ together with a connection $\nabla^{\underline{\mathcal{T}}}$ such that 
\begin{equation}
    S_{c_q}(\nabla^{\mathcal{T}}) = S_{c_q}(\nabla^{\underline{\mathcal{T}}}).
\end{equation}
When $M$ arises as the boundary of a relatively compact strictly pseudoconvex domain $\Omega$, the connection $\nabla^{\underline{\mathcal{T}}}$ admits an ambient construction that is naturally related to the Kähler--Einstein geometry of $\Omega$.
This ambient construction is the key step in comparing our secondary invariants, defined from the CR Cartan geometry of $M$, with characteristic forms arising from the complex geometry of the filling domain $\Omega$, and eventually leads to the bulk--boundary formulas proved below.

\subsection{Reduction}
Let $M$ be a compact strictly pseudoconvex CR manifold of dimension $2n+1$, and assume that $M$ admits a pseudo-Einstein contact form $\theta$.
Let $\sigma_\theta \colon \mathcal{F}_{\theta}(T^{1,0}M) \to \mathcal{G}$ be the reduction of structure group to $\mathit{U}(n)$ described in Subsection \ref{subsec:Exact}.
We introduce another $\mathit{U}(n)$-reduction of the full frame bundle $\mathcal{F}(\mathcal{T})$, denoted by $\sigma'_\theta \colon \mathcal{F}_{\theta}(T^{1,0}M) \to \mathcal{F}(\mathcal{T})$, as follows: 
\begin{equation}
    \sigma'_\theta(\bm{Z}) = \sigma_\theta(\bm{Z}) \cdot \begin{pmatrix}
        1& & \frac{P}{n}\\
         & \delta_{\beta}^{\,\,\,\alpha}& \\
         & & 1
    \end{pmatrix}.
\end{equation}
By \eqref{eq:exact_Weyl}, \eqref{eq:TanakaConn} and \eqref{eq:pseudo_Einstein}, the connection form of $\nabla^{\mathcal{T}}$ with respect to the frame $\sigma'_\theta(\bm{Z})$ is given by 
\begin{equation}\scalebox{0.95}{$
    \begin{pmatrix}
        0& -\sqrt{-1}A_{\beta\rho}\theta^\rho + \sqrt{-1}\kappa_\beta\theta& \kappa_\rho\theta^\rho + \left(\frac{1}{2}\kappa_T+\sqrt{-1}\left(\frac{1}{2n}\Delta_b\kappa-\frac{1}{n}\abs{A}^2\right)\right)\theta\\
        \theta^\alpha& \omega_{\beta}^{\,\,\,\alpha} + \sqrt{-1}\kappa\theta\delta_{\beta}^{\,\,\,\alpha}& \kappa\theta^\alpha - \sqrt{-1}A^{\alpha}_{\,\,\,\conj{\sigma}}\theta^{\conj{\sigma}} + \sqrt{-1}\kappa^\alpha\theta\\
        \sqrt{-1}\theta& -\theta_\beta& \sqrt{-1}\kappa\theta
    \end{pmatrix},$}
    \label{eq:exact_Weyl+pseudo_Einstein}
\end{equation}
where $\kappa = 2P/n$, $\Delta_b\kappa = -\kappa_{\gamma}^{\,\,\,\gamma}-\kappa^{\gamma}_{\,\,\,\gamma}$ and $\abs{A}^2=A_{\alpha\beta}A^{\alpha\beta}$.

Since the transition functions of the frames $\sigma'_\theta(\mathcal{F}_\theta(T^{1,0}M))$ preserve the span of the last $n+1$ components, these components define a subbundle $\underline{\mathcal{T}}\subset\mathcal{T}$.
We define $\nabla^{\underline{\mathcal{T}}}$ to be the connection on $\underline{\mathcal{T}}$ whose connection form, with respect to these frames, is given by the lower-right $(n+1)\times(n+1)$ block of \eqref{eq:exact_Weyl+pseudo_Einstein}.
The first component $Z_0$ of the frames in $\sigma'_\theta(\mathcal{F}_{\theta}(T^{1,0}M))$ is also invariant under the action of $\mathit{U}(n)$ and yields a decomposition 
\begin{equation}
    \mathcal{T} = \langle Z_0 \rangle \oplus \underline{\mathcal{T}}.
    \label{eq:decomposition_T}
\end{equation}
For each $1\le r\le n+1$, an inclusion 
\begin{equation}
    \iota \colon \mathrm{St}(r,\underline{\mathcal{T}}) \hookrightarrow \mathrm{St}(r+1,\mathcal{T})
\end{equation}
is defined by inserting $Z_0$ as a first component.

\begin{thm}
    \label{thm:reduction_theorem}
    For $1\le q\le n+1$, the Cheeger--Simons differential characters associated with $\nabla^{\mathcal{T}}$ and $\nabla^{\underline{\mathcal{T}}}$ agree:
    \begin{equation}
        S_{c_q}(\nabla^{\mathcal{T}}) = S_{c_q}(\nabla^{\underline{\mathcal{T}}}).
    \end{equation}
    Moreover, suppose that $S_{c_q}(\nabla^{\underline{\mathcal{T}}})$ admits an $\mathbb{R}$-valued lift defined by an $(n+2-q)$-frame $\underline{\bm{s}}$ of $\underline{\mathcal{T}}$. 
    Then
    \begin{equation}
        \widetilde S_{c_q}^{\iota(\underline{\bm{s}})}(\nabla^{\mathcal{T}}) = \widetilde S_{c_q}^{\underline{\bm{s}}}(\nabla^{\underline{\mathcal{T}}}).
    \end{equation}
    The same assertion holds for $\mathbb{R}$-valued lifts of products of differential characters.
\end{thm}

\begin{proof}
    First, note that $\iota$ induces a morphism of exact sequences in \eqref{eq:exact_sequence}:
    \begin{equation}
    \scalebox{0.95}{$
        \begin{tikzcd}[ampersand replacement=\&]
        H_{2q-1}(\mathrm{St}(r,\mathbb{C}^{n+1}))
          \arrow[r] \arrow[d, "\wr"']
        \&
        H_{2q-1}(\mathrm{St}(r,\underline{\mathcal{T}}))
          \arrow[r, "\pi_*"] \arrow[d, "\iota_*"']
        \&
        H_{2q-1}(M)
          \arrow[r] \arrow[d, equal]
        \&
        0
        \\
        H_{2q-1}(\mathrm{St}(r+1,\mathbb{C}^{n+2}))
          \arrow[r]
        \&
        H_{2q-1}(\mathrm{St}(r+1,\mathcal{T}))
          \arrow[r, "\pi_*"]
        \&
        H_{2q-1}(M)
          \arrow[r]
        \&
        0
        \end{tikzcd}
        $}
    \end{equation}
    where $\pi$ denotes the projection of the Stiefel bundle.
    For any smooth cycle $z\in Z_{2q-1}(M,\mathbb{Z})$, there exist a smooth cycle $y\in Z_{2q-1}(\mathrm{St}(r,\underline{\mathcal{T}}),\mathbb{Z})$ and a smooth chain $w\in C_{2q}(M,\mathbb{Z})$ such that $z = \pi_* y + \partial w$.
    In the following, forms on $\mathrm{St}(r+1,\mathcal{T})$ are pulled back to $\mathrm{St}(r,\underline{\mathcal{T}})$ via $\iota$ without changing notation.
    Then we have 
    \begin{align}
        S_{c_q}(\nabla^{\mathcal{T}})(z) &= \int_y c_q(\nabla^{\iota\left(\underline{\bm{s}}_{\text{taut}}\right)},\pi^*\nabla^{\mathcal{T}}) + \int_w c_q(\nabla^{\mathcal{T}}) + \mathbb{Z},\\
        S_{c_q}(\nabla^{\underline{\mathcal{T}}})(z) &= \int_y c_q(\nabla^{\underline{\bm{s}}_{\text{taut}}},\pi^*\nabla^{\underline{\mathcal{T}}}) + \int_w c_q(\nabla^{\underline{\mathcal{T}}}) + \mathbb{Z}.
    \end{align}
    If $\underline{\mathcal{T}}$ admits an $r$-frame $\underline{\bm{s}}$, then 
    \begin{align}
        \widetilde{S}_{c_q}^{\iota\left(\underline{\bm{s}}\right)}(\nabla^{\mathcal{T}})(z) &= \int_{\underline{\bm{s}}_*z} c_q(\nabla^{\iota\left(\underline{\bm{s}}_{\text{taut}}\right)},\pi^*\nabla^{\mathcal{T}}),\\
        \widetilde{S}_{c_q}^{\underline{\bm{s}}}(\nabla^{\underline{\mathcal{T}}})(z) &= \int_{\underline{\bm{s}}_*z} c_q(\nabla^{\underline{\bm{s}}_{\text{taut}}},\pi^*\nabla^{\underline{\mathcal{T}}}).
    \end{align}
    Therefore, it suffices to show 
    \begin{equation}
        c_q(\nabla^{\iota\left(\underline{\bm{s}}_{\text{taut}}\right)},\pi^*\nabla^{\mathcal{T}}) = c_q(\nabla^{\underline{\bm{s}}_{\text{taut}}},\pi^*\nabla^{\underline{\mathcal{T}}}) + \text{exact form},
        \label{eq:reduction_differential_form}
    \end{equation}
    since taking its exterior derivative also gives $c_q(\nabla^{\mathcal{T}}) = c_q(\nabla^{\underline{\mathcal{T}}})$.

    For a connection $\nabla$ on $\mathcal{T}$, define the second fundamental forms by 
    \begin{align}
        \II_{\langle Z_0 \rangle}(\nabla) &= p_{\underline{\mathcal{T}}} \circ \nabla|_{\langle Z_0 \rangle} \in A^1\left(M,\Hom(\langle Z_0 \rangle, \underline{\mathcal{T}})\right),\\
        \II_{\underline{\mathcal{T}}}(\nabla) &= p_{\langle Z_0 \rangle} \circ \nabla|_{\underline{\mathcal{T}}} \in A^1\left(M,\Hom(\underline{\mathcal{T}},\langle Z_0 \rangle)\right),
    \end{align}
    where $p_{\langle Z_0 \rangle},p_{\underline{\mathcal{T}}}$ denote the projections onto the indicated subbundles with respect to the decomposition \eqref{eq:decomposition_T}.
    Let $\widetilde{\nabla}^{\mathcal{T}}$ be a connection on $\mathcal{T}$ constructed from $\nabla^{Z_0}$ (the $Z_0$-trivial connection on $\langle Z_0 \rangle$) and  $\nabla^{\underline{\mathcal{T}}}$ with the prescribed second fundamental forms 
    \begin{equation}
        \II_{\langle Z_0 \rangle}(\widetilde{\nabla}^{\mathcal{T}}) = \II_{\langle Z_0 \rangle}(\nabla^{\mathcal{T}}), \quad \II_{\underline{\mathcal{T}}}(\widetilde{\nabla}^{\mathcal{T}}) = 0.
    \end{equation}
    Then the left-hand side of \eqref{eq:reduction_differential_form} can be decomposed as 
    \begin{equation}
        c_q(\nabla^{\iota\left(\underline{\bm{s}}_{\text{taut}}\right)},\pi^*\nabla^{\mathcal{T}}) = c_q(\nabla^{\iota\left(\underline{\bm{s}}_{\text{taut}}\right)},\pi^*\widetilde{\nabla}^{\mathcal{T}}) + c_q(\pi^*\widetilde{\nabla}^{\mathcal{T}},\pi^*\nabla^{\mathcal{T}}) + \text{exact form}.
    \end{equation}

    If we take $\nabla^{\iota\left(\underline{\bm{s}}_{\text{taut}}\right)}=\nabla^{Z_0}\oplus\nabla^{\underline{\bm{s}}_{\text{taut}}}$, the connection form of $(1-t)\nabla^{\iota\left(\underline{\bm{s}}_{\text{taut}}\right)}+t\pi^*\widetilde{\nabla}^{\mathcal{T}}$ with respect to a frame compatible with the decomposition \eqref{eq:decomposition_T} is given by
    \begin{equation}
        \begin{pmatrix}
            0& 0\\
            \bullet& \conn\left((1-t)\nabla^{\underline{\bm{s}}_{\text{taut}}}+t\pi^*\nabla^{\underline{\mathcal{T}}}\right)
        \end{pmatrix},
    \end{equation}
    where $\conn$ denotes the connection form. 
    The curvature form is then 
    \begin{equation}
        \begin{pmatrix}
            0& 0\\
            \bullet& \curv\left((1-t)\nabla^{\underline{\bm{s}}_{\text{taut}}}+t\pi^*\nabla^{\underline{\mathcal{T}}}\right)
        \end{pmatrix}.
    \end{equation}
    It follows from \eqref{eq:relativeCS} that 
    \begin{equation}
        c_q(\nabla^{\iota\left(\underline{\bm{s}}_{\text{taut}}\right)},\pi^*\widetilde{\nabla}^{\mathcal{T}}) = c_q(\nabla^{\underline{\bm{s}}_{\text{taut}}},\pi^*\nabla^{\underline{\mathcal{T}}}).
    \end{equation}

    On the other hand, the connection form of $(1-t)\widetilde{\nabla}^{\mathcal{T}}+t\nabla^{\mathcal{T}}$ with respect to the local frame $\sigma'_\theta(\bm{Z})\,(\bm{Z}\in \mathcal{F}_\theta(T^{1,0}M))$ is given by 
    \begin{equation}\scalebox{0.9}{$
        \begin{pmatrix}
        0& -\sqrt{-1}tA_{\beta\rho}\theta^\rho + \sqrt{-1}t\kappa_\beta\theta& t\kappa_\rho\theta^\rho + t\left(\frac{1}{2}\kappa_T+\sqrt{-1}\left(\frac{1}{2n}\Delta_b\kappa-\frac{1}{n}\abs{A}^2\right)\right)\theta\\
        \theta^\alpha& \omega_{\beta}^{\,\,\,\alpha} + \sqrt{-1}\kappa\theta\delta_{\beta}^{\,\,\,\alpha}& \kappa\theta^\alpha - \sqrt{-1}A^{\alpha}_{\,\,\,\conj{\sigma}}\theta^{\conj{\sigma}} + \sqrt{-1}\kappa^\alpha\theta\\
        \sqrt{-1}\theta& -\theta_\beta& \sqrt{-1}\kappa\theta
    \end{pmatrix}.$}
    \end{equation}
    A direct computation using the structure equations \eqref{eq:str.eq.forTW_1}, \eqref{eq:str.eq.forTW_2} for the Tanaka--Webster connection and the symmetry $A_{\alpha\beta} = A_{\beta\alpha}$ of the torsion shows that the curvature is of the form 
    \begin{equation}
        \begin{pmatrix}
            0& \bullet\\
            0& \bullet
        \end{pmatrix}.
    \end{equation}
    Therefore, by virtue of \eqref{eq:relativeCS}, we have 
    \begin{equation}
        c_q(\pi^*\widetilde{\nabla}^{\mathcal{T}},\pi^*\nabla^{\mathcal{T}})=0.
    \end{equation}
    This proves \eqref{eq:reduction_differential_form}, and hence the assertions for $S_{c_q}$ and its $\mathbb{R}$-valued lifts. 
    The assertion for products follows by wedging the above relative Chern--Simons forms with the remaining characteristic forms.
\end{proof}

\begin{rem}
    Recall that the reduction $\sigma_\theta \colon \mathcal{F}_{\theta}(T^{1,0}M) \to \mathcal{G}$ is given by taking the frame of $\mathcal{T}$ dual to 
    \begin{equation}
        -\frac{1}{n+2}\omega_{\gamma}^{\,\,\,\gamma} + \frac{dc}{c} - \frac{\sqrt{-1}}{n+2}P\theta, \theta^\alpha, \sqrt{-1}\theta.
    \end{equation}
    Then the reduction $\sigma'_\theta \colon \mathcal{F}_\theta(T^{1,0}M) \to \mathcal{F}(\mathcal{T})$ corresponds to the dual frame of 
    \begin{equation}
        -\frac{1}{n+2}\omega_{\gamma}^{\,\,\,\gamma} + \frac{dc}{c} - \frac{\sqrt{-1}}{n(n+2)}R\theta, \theta^\alpha, \sqrt{-1}\theta,
    \end{equation}
    the first component of which is precisely the connection form of $D^\theta$.
    Therefore, the bundle $\underline{\mathcal{T}}$ is the horizontal lift of $T^{\mathbb{C}}M/T^{0,1}M$ with respect to the connection $D^\theta$.
\end{rem}

\subsection{Strictly pseudoconvex domains}
Let $X$ be a complex manifold of complex dimension $n+1$, and let $\Omega \subset X$ be a relatively compact domain with smooth, connected boundary $M = \partial \Omega$.
The boundary $M$ carries a natural CR structure.
We assign an orientation to $TM/HM$ so that a smooth real function $\rho$ on $X$ satisfying 
\begin{equation}
    \Omega = \{\rho<0\},\quad M = \{\rho = 0\},\quad d\rho \ne 0 \text{ on } M
\end{equation}
determines a positive section $\theta = \frac{\sqrt{-1}}{2}(\conj{\partial}-\partial)\rho|_{TM}$ of $(TM/HM)^*=HM^\perp$.
Such a function $\rho$ is called a \emph{defining function} of the domain $\Omega$.
We say $\Omega$ is \emph{strictly pseudoconvex} if the boundary $M$ with the orientation of $TM/HM$ is strictly pseudoconvex.

Let $\Omega\subset X$ be a relatively compact strictly pseudoconvex domain.
Given a defining function $\rho$ of $\Omega$, there is a unique $(1,0)$-vector field $\xi$ near $M$ satisfying 
\begin{equation}
    \xi \rho = 1, \quad \xi \intprod \partial \conj{\partial} \rho = \kappa \conj{\partial} \rho,
    \label{eq:xi}
\end{equation}
where $\kappa$ is a real-valued function called the \emph{transverse curvature} of $\rho$.
Decomposing $\xi = N - \frac{\sqrt{-1}}{2}T$ into real vector fields $N$ and $T$, we obtain 
\begin{equation}
    N\rho = 1,\quad \theta(N) = 0,\quad T\rho = 0,\quad \theta(T)=1,\quad T \intprod d\theta = 0.
    \label{eq:N_and_T}
\end{equation}
Thus $N$ is everywhere transverse to $M = \partial \Omega$ and points outward, while $T$ is the Reeb vector field of the contact form $\theta$.

Let $\mathcal{K}\to X$ be the canonical bundle of $X$.
For integers $w_1,w_2 \in \mathbb{Z}$, define the ambient density bundle by 
\begin{equation}
    \widetilde{\mathcal{E}}(w_1,w_2) = \mathcal{K}^{-w_1/(n+2)} \otimes (\conj{\mathcal{K}})^{-w_2/(n+2)}.
\end{equation}
It is intrinsic and globally defined if $w_1-w_2 \in (n+2)\mathbb{Z}$; otherwise, it is defined only locally.
A section of $\widetilde{\mathcal{E}}(w_1,w_2)$ is identified with a function $\bm{f}$ on $\mathcal{K}^\times$ with the homogeneity condition 
\begin{equation}
    \bm{f}(\lambda \zeta) = \lambda^{w_1/(n+2)}(\conj{\lambda})^{w_2/(n+2)}\bm{f}(\zeta), \quad \zeta \in \mathcal{K}^\times,\lambda \in \mathbb{C}^\times.
\end{equation}
Since $\mathcal{K}|_M = \mathcal{K}_M$, we have $\widetilde{\mathcal{E}}(w_1,w_2)|_M = \mathcal{E}(w_1,w_2)$.

\subsection{Fefferman defining functions}
In this subsection, we introduce a special class of defining functions of relatively compact strictly pseudoconvex domains, called Fefferman defining functions. 
They are characterized as approximate solutions to a complex Monge--Ampère equation. 
Such defining functions need not always exist; their existence is equivalent to the existence of a pseudo-Einstein contact form on the boundary.

Let $\mathcal{K}\to X$ be the canonical bundle of $X$, and let $\bm{\zeta}$ be the tautological $(n+1,0)$-form on $\mathcal{K}$.
Define a volume form on $\mathcal{K}^\times$ by 
\begin{equation}
    \bm{v} = \sqrt{-1}^{(n+2)^2} d\bm{\zeta} \wedge \conj{d\bm{\zeta}}.
\end{equation}

A real density $\bm{\rho} \in \Gamma \widetilde{\mathcal{E}}_{\mathbb{R}}(1,1)$ is called a \emph{defining density} of $\Omega$ if the corresponding homogeneous function on $\mathcal{K}^\times$ is a defining function of $\mathcal{K}^\times|_\Omega$.
\begin{thm}[\cite{hirachi_q-prime_2014}]
    There exists a defining density $\bm{\rho}$ of $\Omega$ satisfying 
    \begin{equation}
        (\sqrt{-1}\partial\conj{\partial}\bm{\rho})^{n+2} = k_n(-1+O(\bm{\rho}^{n+2}))\bm{v},
    \end{equation}
    where $O(\bm{\rho}^{n+2})$ stands for a term $\phi \bm{\rho}^{n+2}$ with $\phi \in \Gamma \widetilde{\mathcal{E}}(-n-2,-n-2)$ and $k_n = (n+1)!/(n+2)$.
    Such a density is unique modulo $O(\bm{\rho}^{n+3})$, and is called a \emph{Fefferman defining density}.
\end{thm}

Now, assume that the boundary $M=\partial\Omega$ admits a pseudo-Einstein contact form $\theta$.
Recall that $\theta$ induces a metric $h_\theta$ on $\mathcal{E}(1,0)$ whose associated metric connection $D^\theta$ is flat; see Theorem \ref{thm:pseudo_Einstein}.
In this setting, we have the following result.

\begin{thm}[\cite{hirachi_variation_2017}]
    \label{thm:Hislop}
    If the boundary $M$ admits a pseudo-Einstein contact form $\theta$, then $h_\theta$ extends to a metric $\widetilde{h}_\theta \in \Gamma\widetilde{\mathcal{E}}(-1,-1)$ on $\widetilde{\mathcal{E}}(1,0)$ that is flat near the boundary.
    The flat connection $D^\theta$ is obtained by restricting the Chern connection of $\widetilde{h}_\theta$.
\end{thm}

For a Fefferman defining density $\bm{\rho}$, the defining function $\rho = \widetilde{h}_\theta\bm{\rho}$ is called the \emph{Fefferman defining function}.
The associated contact form $\frac{\sqrt{-1}}{2}(\conj{\partial}-\partial)\rho$ recovers $\theta$ \cite{farris_intrinsic_1986}.
The Fefferman defining function admits a direct definition in terms of the complex Monge--Ampère equation; see \cite{fefferman_monge-ampere_1976,hislop_cr-invariants_2006}.

\begin{cor}
    If $\widehat{\theta} = e^\Upsilon \theta$ is another pseudo-Einstein contact form, then the function $\widetilde{\Upsilon} = \log(\widetilde{h}_{\widehat{\theta}}/\widetilde{h}_\theta)$ is an extension of $\Upsilon$ over $\Omega$ that is pluriharmonic near the boundary.
    \label{thm:change_of_pseudo_Einstein}
\end{cor}

\subsection{Renormalized connection}
The Monge--Ampère operator and the associated complex Monge--Ampère equation were introduced by Fefferman \cite{fefferman_monge-ampere_1976} in the construction of complete Kähler--Einstein metrics on bounded strictly pseudoconvex domains in complex Euclidean space.
Since such metrics are complete, their Chern connections diverge at the boundary.
Burns and Epstein \cite{burns_characteristic_1990} introduced a renormalization procedure that subtracts a divergent part of the Chern connection.
Marugame \cite{marugame_renormalized_2016,marugame_renormalized_2021} formulated this procedure intrinsically for strictly pseudoconvex domains in complex manifolds.
We briefly review this theory.

Let $X$ be a complex manifold of complex dimension $n+1$, and let $\Omega \subset X$ be a relatively compact strictly pseudoconvex domain whose boundary $M = \partial \Omega$ admits a pseudo-Einstein structure.
Fix a Fefferman defining function $\rho$ of $\Omega$ and set $\theta = \frac{\sqrt{-1}}{2}(\conj{\partial}-\partial)\rho|_{TM}$.
Let $g$ be a Hermitian metric on $\Omega$ that agrees, near the boundary, with the complete Kähler metric $-\sqrt{-1}\partial\conj{\partial}\log(-\rho)$ determined by $\rho$.
Let $\nabla^g$ be the Chern connection of $g$.

The divergent part of $\nabla^g$ is described by 
\begin{equation}
    Y_{i}^{\,\,\,j} = -\frac{1}{\rho}(\delta_{i}^{\,\,\,j}\rho_k + \delta_{k}^{\,\,\,j}\rho_i)\theta^k,
\end{equation}
where $(\theta^i)$ is an arbitrary $(1,0)$-coframe of $\Omega$ and $\partial\rho = \rho_i\theta^i$.
Then the connection $\conj{\nabla}^g = \nabla^g - Y$ on $T^{1,0}\Omega$ extends smoothly to a connection on $T^{1,0}\conj{\Omega}|_M$, and is called the \emph{renormalized connection} associated with $\rho$.
One can also describe the divergent part of the curvature tensor $R^g$ of $\nabla^g$ by 
\begin{equation}
    K_{i}^{\,\,\,j} = -(\delta_{i}^{\,\,\,j}g_{k\conj{l}} + \delta_{k}^{\,\,\,j}g_{i\conj{l}})\theta^k \wedge \theta^{\conj{l}}.
\end{equation}
The endomorphism-valued $2$-form $W = R^g - K$ also extends smoothly to the boundary, and is called the \emph{renormalized curvature} associated with $\rho$.

The renormalized curvature form $W_{i}^{\,\,\,j}$ and the curvature form $\Theta_{i}^{\,\,\,j}$ of the renormalized connection are related by 
\begin{equation}
    \Theta_{i}^{\,\,\,j} = W_{i}^{\,\,\,j} + u_i \wedge \theta^j, \quad u_i = \frac{\rho_{ik}}{\rho}\theta^k,
    \label{eq:renormalized_curvatures}
\end{equation}
where the subscripts on $\rho$ denote covariant derivatives by $\conj{\nabla}^g$.
In particular, $W_{i}^{\,\,\,j}$ is exactly the $(1,1)$-part of $\Theta_{i}^{\,\,\,j}$.

Since both $R^g$ and $K$ satisfy the usual Kähler symmetry relations, the same holds for $W$.
Moreover, the trace of $W$ vanishes to high order near the boundary \cite{marugame_renormalized_2016}:
\begin{equation}
\tr_g W = \Ric (g) + (n+2)g = O(\rho^n).
\end{equation}
Equivalently, the complete Kähler metric $g$ is approximately Einstein near the boundary.

\subsection{Ambient metric construction}
When a CR manifold arises as the boundary of a strictly pseudoconvex domain, the normal tractor connection can be recovered from the complex geometry of the domain \cite{hirachi_q-prime_2014,case_p-operator_2020}.
This method is known as the \emph{ambient metric construction}.
In the present setting, the flat connection $D^\theta$ also admits an ambient construction: it is the restriction to the boundary of the Chern connection of the flat Hermitian metric $\widetilde{h}_\theta$ on $\widetilde{\mathcal{E}}(1,0)$. 
Hence the tensor product connection
\begin{equation}
    \nabla^{\mathcal{T}}
    =
    D^\theta\otimes\nabla^{\mathcal{T}^{\#}}
\end{equation}
admits an ambient construction as well.

For our later use of obstruction theory, we need a more precise form of this construction. 
Namely, we need a simultaneous ambient construction of the pair of connections $(\nabla^{\mathcal{T}},\nabla^{\underline{\mathcal{T}}})$.
Following Marugame \cite[Section 6]{marugame_renormalized_2021}, we give a direct computational proof that identifies this pair with the boundary restriction of a natural pair of connections over $\conj{\Omega}$.

Let $X$ be a complex manifold of complex dimension $n+1$, and let $\Omega \subset X$ be a relatively compact strictly pseudoconvex domain whose boundary $M = \partial \Omega$ admits a pseudo-Einstein structure.
Fix a Fefferman defining function $\rho$ of $\Omega$ and set $\theta = \frac{\sqrt{-1}}{2}(\conj{\partial}-\partial)\rho|_{TM}$.
Let $\widetilde{h}_\theta \in \Gamma \widetilde{\mathcal{E}}(-1,-1)$ be the metric on $\widetilde{\mathcal{E}}(1,0)$ given by Theorem \ref{thm:Hislop}.

We construct vector bundles over $\conj{\Omega}$ that extend $\mathcal{T}$ and $\underline{\mathcal{T}}$ from the boundary.
Consider the horizontal lift $\Ker \partial \log \widetilde{h}_\theta \subset T^{1,0}\widetilde{\mathcal{E}}(-1,0)^\times$ of $T^{1,0}\conj{\Omega}$ with respect to the Chern connection of $\widetilde{h}_\theta$, which we also denote by $T^{1,0}\conj{\Omega}$.
Let $\widetilde{Z}_0$ be the vertical $(1,0)$-vector field on $\widetilde{\mathcal{E}}(-1,0)^\times$ such that $\partial \log \widetilde{h}_\theta(\widetilde{Z}_0) = 1$.
Then we have a decomposition 
\begin{equation}
    T^{1,0}\widetilde{\mathcal{E}}(-1,0)^\times = \langle \widetilde{Z}_0 \rangle \oplus T^{1,0}\conj{\Omega}.
    \label{eq:decomposition_Ttilde}
\end{equation}

Let $\widetilde{\mathcal{T}}\to\conj{\Omega}$ be the vector bundle whose sections are the $\mathbb{C}^{\times}$-invariant sections of $T^{1,0}\widetilde{\mathcal{E}}(-1,0)^\times$.
Since the choice of $(n+2)$nd root $\widetilde{\mathcal{E}}(-1,0)$ of the canonical bundle $\mathcal{K}$ is determined up to the action of $\{ c \in \mathbb{C}^\times \mid c^{n+2} = 1 \}$, the bundle $\widetilde{\mathcal{T}}$ is globally well-defined.
Moreover, the above decomposition of $T^{1,0}\widetilde{\mathcal{E}}(-1,0)^\times$ is clearly $\mathbb{C}^{\times}$-invariant and therefore descends to a decomposition of $\widetilde{\mathcal{T}}$.
For a local unitary frame $\bm{Z} \in \mathcal{F}_\theta(T^{1,0}M)$, $(\widetilde{Z}_0,\bm{Z},\xi)$ is a local frame of $\widetilde{\mathcal{T}}|_M$, where $\xi$ is the $(1,0)$-vector field determined by \eqref{eq:xi}.

Let $\bm{\rho} = \widetilde{h}_\theta^{-1}\rho$ be the Fefferman defining density and consider a Lorentz--Hermitian metric $\widetilde{h}$ on $\widetilde{\mathcal{T}}$ that agrees, near the boundary, with $\sqrt{-1}\widetilde{h}_\theta\partial\conj{\partial}\bm{\rho}$.
Let $\nabla^{\widetilde{h}}$ be the corresponding Chern connection.

\begin{thm}[cf. {\cite[Theorem 6.4]{marugame_renormalized_2021}}]
    \label{thm:ambient_construction}
    The map sending the local frame $(\widetilde{Z}_0,\bm{Z},\xi)$ of $\widetilde{\mathcal{T}}|_M$ to the local frame $\sigma'_\theta(\bm{Z})$ of $\mathcal{T}$ defines a well-defined isomorphism $\widetilde{\mathcal{T}}|_M \simeq \mathcal{T}$.
    This isomorphism preserves the decompositions \eqref{eq:decomposition_Ttilde} and \eqref{eq:decomposition_T}.
    Under this isomorphism, $\nabla^{\mathcal{T}}$ corresponds to the restriction of  $\nabla^{\widetilde{h}}$, and $\nabla^{\underline{\mathcal{T}}}$ corresponds to the restriction of $\conj{\nabla}^g$.
\end{thm}

\begin{proof}
    It is clear that the map gives a well-defined isomorphism preserving the decompositions; it remains to show that it identifies the connections.
    By the same computation as in the proof of \cite[Proposition 4.2]{marugame_renormalized_2021}, the connection form of $\nabla^{\widetilde{h}}$ with respect to the frame $(\widetilde{Z}_0,\bm{Z},\xi)$ around a boundary point is given by 
    \begin{equation}
        \begin{pmatrix}
            0& u_\beta& u_{n+1}\\
            \theta^\alpha& \theta_{\beta}^{\,\,\,\alpha}& \theta_{n+1}^{\,\,\,\alpha}\\
            \partial \rho& \theta_{\beta}^{\,\,\,n+1}& \theta_{n+1}^{\,\,\,n+1}
        \end{pmatrix},
    \end{equation}
    where the forms $u_i$ are given in \eqref{eq:renormalized_curvatures} and $\theta_{i}^{\,\,\,j}$ denotes the connection form of the renormalized connection $\conj{\nabla}^g$.
    Using \cite[Proposition 3.5 and (4.12)--(4.15)]{marugame_renormalized_2016},
    we see that this matrix coincides with the connection form \eqref{eq:exact_Weyl+pseudo_Einstein} of $\nabla^{\mathcal{T}}$ with respect to the frame $\sigma'_\theta(\bm{Z})$. 
    It follows that $\nabla^{\mathcal{T}}$ and $\nabla^{\underline{\mathcal{T}}}$ are the restrictions of $\nabla^{\widetilde{h}}$ and $\conj{\nabla}^g$, respectively.
\end{proof}

\subsection{Proof of CR invariance}
We are now ready to prove the CR invariance of the CR Cheeger--Simons invariants using the ambient metric construction.
\begin{proof}[Proof of Theorem \ref{thm:CR_invariance}]
    Let $\widehat{\theta}$ be another pseudo-Einstein contact form and set $\widehat{\nabla}^{\mathcal{T}} = D^{\widehat{\theta}}\otimes \nabla^{\mathcal{T}^\#}$.
    Let us temporarily assume that $\mathcal{T}$ admits an $r_1=((n+2)-q_1+1)$-frame $\bm{s}_1$.
    Using \eqref{eq:productCS} and \eqref{eq:difference_form}, we have 
    \begin{equation}
    \begin{aligned}
        &\int_M c_{q_1}(\nabla^{\bm{s}_1},\widehat{\nabla}^{\mathcal{T}})\wedge(c_{q_2}\cdots c_{q_m})(\widehat{\nabla}^{\mathcal{T}}) - \int_M c_{q_1}(\nabla^{\bm{s}_1},\nabla^{\mathcal{T}})\wedge(c_{q_2}\cdots c_{q_m})(\nabla^{\mathcal{T}})\\
        &= \int_M (c_{q_1}\cdots c_{q_m})(\nabla^{\mathcal{T}},\widehat{\nabla}^{\mathcal{T}}).
    \end{aligned}
    \end{equation}
    Hence 
    \begin{equation}
    \begin{aligned}
        &\left(S_{c_{q_1}}(\widehat{\nabla}^{\mathcal{T}})*\cdots*S_{c_{q_m}}(\widehat{\nabla}^{\mathcal{T}})\right)(M) - \left(S_{c_{q_1}}(\nabla^{\mathcal{T}})*\cdots*S_{c_{q_m}}(\nabla^{\mathcal{T}})\right)(M)\\
        &= \int_M (c_{q_1}\cdots c_{q_m})(\nabla^{\mathcal{T}},\widehat{\nabla}^{\mathcal{T}}) + \mathbb{Z},\\
        &\left(\widetilde{S}^{\bm{s}_1}_{c_{q_1}}(\widehat{\nabla}^{\mathcal{T}})*\cdots*S_{c_{q_m}}(\widehat{\nabla}^{\mathcal{T}})\right)(M) - \left(\widetilde{S}^{\bm{s}_1}_{c_{q_1}}(\nabla^{\mathcal{T}})*\cdots*S_{c_{q_m}}(\nabla^{\mathcal{T}})\right)(M)\\
        &= \int_M (c_{q_1}\cdots c_{q_m})(\nabla^{\mathcal{T}},\widehat{\nabla}^{\mathcal{T}}).
    \end{aligned}
    \end{equation}
    The first formula remains valid even when $\mathcal{T}$ admits no $r_1$-frame; see \cite[Proposition 2.9]{cheeger_differential_1985}.
    Thus it remains to show that  
    \begin{equation}
        \int_M (c_{q_1}\cdots c_{q_m})(D^\theta \otimes \nabla^{\mathcal{T}^\#}, D^{\widehat{\theta}} \otimes \nabla^{\mathcal{T}^\#}) = 0.
    \end{equation}

    Set $\Phi = c_{q_1}\cdots c_{q_m}$.
    Let $\eta,\widehat{\eta}$ be the connection forms of $D^{\theta},D^{\widehat{\theta}}$, respectively. 
    Since $D^{\theta},D^{\widehat{\theta}}$ are flat, the curvature $\Omega^{(t)}$ of the homotopy of connections $\nabla^{(t)} = \left((1-t)D^\theta  + tD^{\widehat{\theta}}\right) \otimes \nabla^{\mathcal{T}^\#}$ coincides with the curvature $\Omega^{\mathcal{T}^\#}$ of $\nabla^{\mathcal{T}^\#}$.
    Therefore, by \eqref{eq:relativeCS}, 
    \begin{equation}
        \Phi(D^\theta \otimes \nabla^{\mathcal{T}^\#}, D^{\widehat{\theta}} \otimes \nabla^{\mathcal{T}^\#}) = (n+1)\widetilde{\Phi}\left((\widehat{\eta}-\eta)\delta_{B}^{\,\,\,A},\Omega^{\mathcal{T}^\#},\ldots,\Omega^{\mathcal{T}^\#}\right).
    \end{equation}
    When $\dim M=3$, the possible Chern monomials are $c_1\cdot c_1$ and $c_2$, and one computes  
    \begin{align}
        \widetilde{c_1\cdot c_1}\left((\widehat{\eta}-\eta)\delta_{B}^{\,\,\,A},\Omega^{\mathcal{T}^\#}\right) &= c_1\left((\widehat{\eta}-\eta)\delta_{B}^{\,\,\,A}\right) \wedge c_1(\Omega^{\mathcal{T}^\#}) = 0,\\
        \widetilde{c_2}\left((\widehat{\eta}-\eta)\delta_{B}^{\,\,\,A},\Omega^{\mathcal{T}^\#}\right) &= \frac{\sqrt{-1}}{2\pi}(\widehat{\eta}-\eta) \wedge c_1(\Omega^{\mathcal{T}^\#}) = 0,
    \end{align}
    so our assertion holds in this case.

    When $\dim M\ge 5$, Lempert's theorem realizes $M$ as the boundary of a relatively compact strictly pseudoconvex domain $\Omega$ in a Kähler manifold $X$ \cite[Theorem 8.1]{lempert_algebraic_1995}.
    Since $X$ is Kähler, the function $\widetilde{\Upsilon} = \log(\widetilde{h}_{\widehat{\theta}}/\widetilde{h}_\theta)$ from Corollary \ref{thm:change_of_pseudo_Einstein} can be taken to be pluriharmonic on the whole domain $\Omega$ \cite[Theorem 7.1]{hirachi_q-prime_2014}.
    Since $D^{\theta}$ and $D^{\widehat{\theta}}$ are the restrictions of the Chern connections of $(\widetilde{\mathcal{E}}(1,0),\widetilde{h}_\theta)$ and $(\widetilde{\mathcal{E}}(1,0),\widetilde{h}_{\widehat{\theta}})$, respectively, we have $\widehat{\eta}-\eta = \partial\log\widetilde{h}_{\widehat{\theta}}|_M - \partial\log\widetilde{h}_\theta|_M = \partial \widetilde{\Upsilon}|_M$.
    Hence 
    \begin{equation}
        \Phi(D^\theta \otimes \nabla^{\mathcal{T}^\#}, D^{\widehat{\theta}} \otimes \nabla^{\mathcal{T}^\#}) = \partial \widetilde{\Upsilon} \wedge \Psi(\nabla^{\widetilde{h}})|_M
    \end{equation}
    for some invariant polynomial $\Psi$ of the appropriate degree.
    Since $\widetilde{\Upsilon}$ is pluriharmonic on $\Omega$ and $\Psi(\nabla^{\widetilde{h}})$ is closed, Stokes' theorem gives
    \begin{equation}
        \int_M \Phi(D^\theta \otimes \nabla^{\mathcal{T}^\#}, D^{\widehat{\theta}} \otimes \nabla^{\mathcal{T}^\#}) = \int_\Omega \conj{\partial}\partial \widetilde{\Upsilon} \wedge \Psi(\nabla^{\widetilde{h}}) = 0.
    \end{equation}
\end{proof}

\subsection{Residues of Chern classes}
We now consider the following general situation. 
Let $\conj{\Omega}$ be a compact oriented manifold with boundary $M=\partial\Omega$, and let $\widetilde{E}\to\conj{\Omega}$ be a complex vector bundle of rank $k$ equipped with a connection $\widetilde{\nabla}$. 
We denote their restrictions to the boundary by $E=\widetilde{E}|_M, \nabla=\widetilde{\nabla}|_E$.
We assume $\dim \Omega = 2(n+1)$ and consider the differential character $S_{c_{q_1}}(\nabla)*\cdots*S_{c_{q_m}}(\nabla)$ on the boundary, where $q_1 + \cdots + q_m = n+1$.
Then its evaluation on the fundamental cycle $M$ is known to coincide with the characteristic number of the domain \cite{cheeger_differential_1985}:
\begin{equation}
    \left(S_{c_{q_1}}(\nabla)*\cdots*S_{c_{q_m}}(\nabla)\right)(M) = \int_\Omega (c_{q_1} \cdots c_{q_m})(\widetilde{\nabla}) + \mathbb{Z}.
\end{equation}
If the differential character admits an $\mathbb{R}$-valued lift defined by an $r_1=k-q_1+1$ frame $\bm{s}_1$ of $E$, then the difference
\begin{equation}
    \left(\widetilde{S}^{\bm{s}_1}_{c_{q_1}}(\nabla)*S_{c_{q_2}}(\nabla)*\cdots*S_{c_{q_m}}(\nabla)\right)(M) - \int_\Omega (c_{q_1} \cdots c_{q_m})(\widetilde{\nabla}) \in \mathbb{Z}
\end{equation}
is a topological quantity coming from $\bm{s}_1$, which we describe in this subsection.

Let $K$ be a smooth triangulation of $\conj{\Omega}$ compatible with the boundary $M$, and let $K^*$ be its dual cellular decomposition.
Since the homotopy groups of the Stiefel manifold $\mathrm{St}(r_1,\mathbb{C}^k)$ vanish up to dimension $2q_1-2$, $\bm{s}_1$ extends to an $r_1$-frame of $\widetilde{E}$ over the $(2q_1-1)$-skeleton $(K^*)^{2q_1-1}$.
Let $\sigma\in K$ be a $2(n-q_1+1)$-simplex, and let $\sigma^* \in K^*$ be its dual $2q_1$-cell.
Then the $r_1$-frame $\bm{s}_1$, already extended over $\partial(\sigma^*)$, induces a map
\begin{align}
    \varphi \colon \mathbb{S}^{2q_1-1} \simeq \partial(\sigma^*)
    &\stackrel{\bm{s}_1}{\longrightarrow} \mathrm{St}(r_1,\widetilde{E})|_{\sigma^*}\\
    &\simeq \sigma^* \times \mathrm{St}(r_1,\mathbb{C}^k)
    \twoheadrightarrow \mathrm{St}(r_1,\mathbb{C}^k),
\end{align}
where $\mathrm{St}(r_1,\widetilde{E})$ denotes the Stiefel bundle of $r_1$-frames of $\widetilde{E}$, which is trivial over the contractible cell $\sigma^*$.
Since $\pi_{2q_1-1}(\mathrm{St}(r_1,\mathbb{C}^k))$ is isomorphic to $\mathbb{Z}$ and admits a canonical generator $x_{2q_1-1}$, the degree of $\varphi$ is an integer; we denote it by $I(\bm{s}_1,\sigma^*)$.
The class 
\begin{equation}
    \Res_{c_{q_1}}(\bm{s}_1,\widetilde{E}) = \left[\sum_{\sigma \colon 2(n-q_1+1)\text{-simplex of }K}I(\bm{s}_1,\sigma^*)\sigma\right] \in H_{2(n-q_1+1)}(\conj{\Omega},\mathbb{Z}),
\end{equation}
which is independent of the triangulation $K$ of $\conj{\Omega}$ and the extension of $\bm{s}_1$ over $(K^*)^{2q_1-1}$, is called the \emph{residue of Chern class} $c_{q_1}$ defined by $\bm{s}_1$.

This residue corresponds to the relative Chern class via Lefschetz duality.
Consider a cochain complex $(A^\bullet(\conj{\Omega},M),d)$ defined by 
\begin{equation}
    A^p(\conj{\Omega},M) = A^p(\conj{\Omega}) \oplus A^{p-1}(M),\quad d(\omega,\eta) = (d\omega,\omega|_M - d\eta).
\end{equation}
Its cohomology $H^\bullet_{\text{dR}}(\conj{\Omega},M)$ is called the \emph{relative de Rham cohomology}.
For a cohomology class $\alpha = [(\omega,\eta)] \in H^p_{\text{dR}}(\conj{\Omega},M)$, its \emph{Lefschetz dual} $L(\alpha) \in H_{2(n+1)-p}(\conj{\Omega},\mathbb{C})$ is characterized by 
\begin{equation}
    \int_\Omega \omega \wedge \tau - \int_M \eta \wedge \tau = \int_{L(\alpha)} \tau
\end{equation}
for every closed differential form $\tau \in A^{2(n+1)-p}(\conj{\Omega})$.

The $r_1$-frame $\bm{s}_1$ of $E$ determines the \emph{relative Chern class}
\begin{equation}
    c_{q_1}(\bm{s}_1,\widetilde{E}) = [(c_{q_1}(\widetilde{\nabla}),c_{q_1}(\nabla^{\bm{s}_1},\nabla))] \in H^{2q_1}_{\text{dR}}(\conj{\Omega},M).
\end{equation}
Suwa's result \cite{suwa_residue_2008} gives 
\begin{equation}
    L\left(c_{q_1}(\bm{s}_1,\widetilde{E})\right) = \gamma\left(\Res_{c_{q_1}}(\bm{s}_1,\widetilde{E})\right),
\end{equation}
where $\gamma$ denotes the natural map $\gamma \colon H_\bullet(\conj{\Omega},\mathbb{Z}) \to H_\bullet(\conj{\Omega},\mathbb{C})$.
See \cite{suwa_complex_2024} for a thorough treatment.
As a corollary, if $\Phi = c_{q_2}\cdots c_{q_m}$, then we obtain the \emph{residue formula} 
\begin{small}
\begin{equation}
    \int_\Omega (c_{q_1} \cdot \Phi)(\widetilde{\nabla}) = \left\langle\Res_{c_{q_1}}(\bm{s}_1,\widetilde{E}),\Phi(\widetilde{E})\right\rangle + \left(\widetilde{S}^{\bm{s}_1}_{c_{q_1}}(\nabla)*S_{c_{q_2}}(\nabla)*\cdots*S_{c_{q_m}}(\nabla)\right)(M).
    \label{eq:residue_formula}
\end{equation}
\end{small}

\subsection{Bulk-boundary formulas}
We return to the CR setting above. 
Thus $X$ is a complex manifold of complex dimension $n+1$, and $\Omega\subset X$ is a relatively compact strictly pseudoconvex domain whose boundary $M=\partial\Omega$ admits a pseudo-Einstein structure. 
We fix a Fefferman defining function $\rho$ of $\Omega$ and set $\theta=\frac{\sqrt{-1}}{2}(\conj{\partial}-\partial)\rho|_{TM}$.
By the reduction theorem (Theorem \ref{thm:reduction_theorem}), the ambient construction (Theorem \ref{thm:ambient_construction}) and the residue formula \eqref{eq:residue_formula}, the CR Cheeger--Simons invariants can be expressed as integrals over $\Omega$:
\begin{equation}
    \mu_{c_{q_1}\cdots c_{q_m}}(M) = \int_\Omega (c_{q_1}\cdots c_{q_m})(\nabla^{\widetilde{h}})+\mathbb{Z} = \int_\Omega (c_{q_1}\cdots c_{q_m})(\conj{\nabla}^g) + \mathbb{Z}.
\end{equation}
When $\underline{\mathcal{T}} = T^{1,0}\conj{\Omega}|_M$ admits an $r_1=((n+1)-q_1+1)$-frame $\underline{\bm{s}}_1$, setting $\Phi = c_{q_2}\cdots c_{q_m}$, we have 
\begin{equation}
\begin{aligned}
    &\left(\widetilde{S}_{c_{q_1}}^{\iota\left(\underline{\bm{s}}_1\right)}(\nabla^{\mathcal{T}}) * S_{c_{q_2}}(\nabla^{\mathcal{T}}) * \cdots * S_{c_{q_m}}(\nabla^{\mathcal{T}})\right)(M)\\
    &=\left(\widetilde{S}_{c_{q_1}}^{\underline{\bm{s}}_1}(\nabla^{\underline{\mathcal{T}}}) * S_{c_{q_2}}(\nabla^{\underline{\mathcal{T}}}) * \cdots * S_{c_{q_m}}(\nabla^{\underline{\mathcal{T}}})\right)(M)\\
    &= \int_\Omega (c_{q_1}\cdot \Phi)(\nabla^{\widetilde{h}}) - \left\langle\Res_{c_{q_1}}(\iota\left(\underline{\bm{s}}_1\right),\widetilde{\mathcal{T}}),\Phi(\widetilde{\mathcal{T}})\right\rangle\\
    &= \int_\Omega (c_{q_1}\cdot \Phi)(\conj{\nabla}^g) - \left\langle\Res_{c_{q_1}}(\underline{\bm{s}}_1,T^{1,0}\conj{\Omega}),\Phi(T^{1,0}\conj{\Omega})\right\rangle.
\end{aligned}
\end{equation}
In particular, when $X=\mathbb{C}^{n+1}$, the bundle $T^{1,0}\conj{\Omega}$ admits a global frame.
Using this frame to define the lift of the differential characters, we recover  \cite[Theorem 5.2 (a)]{burns_characteristic_1990}: 
\begin{equation}
    \widetilde{\mu}_{c_{q_1}\cdots c_{q_m}}(M) = \int_\Omega (c_{q_1}\cdots c_{q_m})(\conj{\nabla}^g) \quad (m\ge 2).
\end{equation}

The bundle $T^{1,0}\conj{\Omega}|_M$ always admits a $1$-frame, namely the $(1,0)$-vector field $\xi$ determined by \eqref{eq:xi}.
Under the isomorphism $\widetilde{\mathcal{T}}|_M \simeq \mathcal{T}$ in Theorem \ref{thm:ambient_construction}, the induced $2$-frame $\iota(\xi)$ corresponds to the canonical $2$-frame used to lift the differential character $S_{c_{n+1}}(\nabla^{\mathcal{T}})$ in Subsection \ref{subsec:Secondary}.
Since the real part $N$ of $\xi$ points outward, the Poincaré--Hopf theorem gives 
\begin{equation}
    \left\langle\Res_{c_{n+1}}(\xi,T^{1,0}\conj{\Omega}),1\right\rangle = \chi(\Omega).
\end{equation}
Therefore, we recover the formulas of \cite[Theorem 5.2 (b)]{burns_characteristic_1990} and \cite{marugame_renormalized_2016}: 
\begin{equation}
    \int_\Omega c_{n+1}(\conj{\nabla}^g) = \chi(\Omega) + \widetilde{\mu}_{c_{n+1}}(M).
    \label{eq:renormalized_CGB}
\end{equation}
This formula is often referred to as the \emph{renormalized Chern--Gauss--Bonnet formula}.

\section{Examples}
In this section, we compute CR Cheeger--Simons invariants for boundaries of Reinhardt domains and for regular Sasakian $\eta$-Einstein manifolds.
In these cases, the invariant $\widetilde{\mu}_{c_{n+1}}$ associated with the top Chern polynomial was computed by Marugame \cite{marugame_renormalized_2016} and Takeuchi \cite{takeuchi_formulae_2024}.
We therefore focus on the invariants associated with products of Chern polynomials.

\subsection{Boundaries of Reinhardt domains}
For $r>0$, consider the Reinhardt domain 
\begin{equation}
    \Omega_r = \left\{(z^1,\ldots,z^{n+1})\in (\mathbb{C}^{\times})^{n+1} \mid \sum_{i=1}^{n+1} (\log \abs{z^i})^2 < r^2 \right\}.
\end{equation}
We denote its boundary by $M_r = \partial\Omega_r$. 
Using the projection 
\begin{equation}
    \pi \colon M_r \ni (z^1,\ldots,z^{n+1}) \mapsto (\log \abs{z^1},\ldots,\log \abs{z^{n+1}}) \in \mathbb{S}_r^n,
\end{equation}
one sees that the boundary $M_r$ is diffeomorphic to $\mathbb{S}_r^{n}\times\mathbb{T}^{n+1}$ and that $T^{1,0}M_r \simeq \pi^*T^{\mathbb{C}}\mathbb{S}_r^n$.

Since the complexified tangent bundle $T^{\mathbb{C}}\mathbb{S}_r^n$ is trivial, the generalized Burns--Epstein invariants are defined for $M_r$.
To exploit the symmetry of $M_r$, however, we use the global trivialization $(z^i\partial/\partial z^i)$ of $T^{1,0}(\mathbb{C}^\times)^{n+1}|_{M_r}\simeq\underline{\mathcal{T}}$, and compute the associated lift of $S_{c_{q_1}}(\nabla^{\underline{\mathcal{T}}})*\cdots*S_{c_{q_m}}(\nabla^{\underline{\mathcal{T}}})$.

Let 
\begin{equation}
    \rho = 2\left(\sum_{i=1}^{n+1} (\log \abs{z^i})^2 - r^2\right)
\end{equation}
be the natural defining function of $\Omega_r$.
As will be clear from the computation below, the associated contact form $\theta = \frac{\sqrt{-1}}{2}(\conj{\partial}-\partial)\rho$ is pseudo-Einstein.
In local logarithmic coordinates $w^i=\log z^i$, the holomorphic transformations 
\begin{equation}
    F_{A,b}(w)=Aw+\sqrt{-1}b, \quad A\in \mathit{O}(n+1),\,b\in\mathbb{R}^{n+1},
\end{equation}
generate a transitive pseudogroup of local CR diffeomorphisms on $M_r$.
Although the transformations corresponding to general $A\in \mathit{O}(n+1)$ do not descend to globally defined transformations of $(\mathbb{C}^\times)^{n+1}$, they are well defined locally after choosing branches of the logarithms.
The pseudogroup preserves the defining function $\rho$, the contact form $\theta$, and hence the reduced connection $\nabla^{\underline{\mathcal{T}}}$. 
Moreover, since the action induces constant gauge transformations 
\begin{equation}
    (F_{A,b})_*\frac{\partial}{\partial w^i} = A_i^{\,\,\,j}\frac{\partial}{\partial w^j}
\end{equation}
on the global frame $\bm{s}=(z^i\partial/\partial z^i)=(\partial/\partial w^i)$, the lift of the Cheeger--Simons differential character $\widetilde{S}^{\bm{s}}_{c_q}(\nabla^{\underline{\mathcal{T}}})$ is also preserved.
It therefore suffices to compute $\widetilde{S}^{\bm{s}}_{c_q}(\nabla^{\underline{\mathcal{T}}})$ at a single point, which we take to be $p=(1,\ldots,1,e^r)$.

Fix logarithmic coordinates $w^i = x^i + \sqrt{-1}y^i = \log z^i$ around $p$, and set $\eta^i = dw^i = dz^i/z^i$.
An admissible coframe around $p$ is given by setting 
\begin{equation}
    \theta^{\alpha} = \eta^\alpha - \frac{\sum_\gamma h^{\alpha\conj{\gamma}}x^\gamma}{2(x^{n+1})^2}\partial \rho,
\end{equation}
where 
\begin{equation}
    h_{\alpha\conj{\beta}} = \delta_{\alpha\conj{\beta}} + \frac{x^\alpha x^\beta}{(x^{n+1})^2}.
\end{equation}
Marugame \cite[Proposition 5.2]{marugame_renormalized_2016} showed that the Tanaka--Webster connection of $\theta$ satisfies 
\begin{align}
    \omega_{\beta}^{\,\,\,\alpha} &= -\frac{\sqrt{-1}}{4r^2}\theta\delta_{\beta}^{\,\,\,\alpha},\\
    A_{\alpha\beta} &= -\frac{\sqrt{-1}}{4r^2}\delta_{\alpha\beta},\\
    R_{\beta\,\,\,\gamma\conj{\delta}}^{\,\,\,\alpha} &= \frac{1}{4r^2}(\delta_{\beta}^{\,\,\,\alpha}\delta_{\gamma\conj{\delta}} + \delta_{\gamma}^{\,\,\,\alpha}\delta_{\beta\conj{\delta}} - \delta_{\conj{\delta}}^{\,\,\,\alpha}\delta_{\beta\gamma})
\end{align}
around $p$.
Using this, we see that the connection form $\theta_i^{\,\,\,j}$ of $\nabla^{\underline{\mathcal{T}}}$ satisfies 
\begin{align}
    \theta_{\beta}^{\,\,\,\alpha} &= -\frac{\sqrt{-1}}{4r^2(n+1)}\theta\delta_{\beta}^{\,\,\,\alpha} + \frac{1}{2r^2}x^\alpha(\theta^\beta + \theta^{\conj{\beta}}),\\
    \theta_{\beta}^{\,\,\,n+1} &= -\theta^{\conj{\beta}},\\
    \theta_{n+1}^{\,\,\,\alpha} &= \frac{n}{4r^2(n+1)}\theta^\alpha + \frac{1}{4r^2}\theta^{\conj{\alpha}},\\
    \theta_{n+1}^{\,\,\,n+1} &= \frac{\sqrt{-1}n}{4r^2(n+1)}\theta
\end{align}
and the curvature form $\Theta_i^{\,\,\,j}$ satisfies 
\begin{align}
    \Theta_{\beta}^{\,\,\,\alpha} &= \frac{1}{4r^2(n+1)}\delta_{\beta}^{\,\,\,\alpha}\delta_{\gamma\conj{\delta}}\theta^\gamma \wedge \theta^{\conj{\delta}}\\ &\qquad + \frac{1}{4r^2(n+1)}\theta^\alpha \wedge \theta^{\conj{\beta}} + \frac{1}{4r^2}(\theta^\alpha + \theta^{\conj{\alpha}}) \wedge \theta^{\conj{\beta}},\\
    \Theta_{\beta}^{\,\,\,n+1} &= -\frac{\sqrt{-1}}{4r^2}\theta^\beta \wedge \theta,\\
    \Theta_{n+1}^{\,\,\,\alpha} &= \frac{\sqrt{-1}}{16r^4}\theta^\alpha \wedge \theta + \frac{\sqrt{-1}(n+2)}{16r^4(n+1)}\theta^{\conj{\alpha}} \wedge \theta,\\
    \Theta_{n+1}^{\,\,\,n+1} &= 0.
\end{align}

\subsubsection{Computation of the CR Cheeger--Simons invariant}
We compute the CR Cheeger--Simons invariant
\begin{equation}
    \widetilde{\mu}_{c_{q_1}\cdots c_{q_m}}(M_r) = \int_{M_r} c_{q_1}(\nabla^{\bm{s}},\nabla^{\underline{\mathcal{T}}})\wedge(c_{q_2}\cdots c_{q_m})(\nabla^{\underline{\mathcal{T}}}).
\end{equation}
For computational convenience, we use Chern characters 
\begin{equation}
    \ch_q(A)=\frac{1}{q!}\tr\left(\frac{\sqrt{-1}}{2\pi}A\right)^q
\end{equation}
instead of Chern polynomials, and compute 
\begin{equation}
    \int_{M_r}\ch_{q_1}(\nabla^{\bm{s}},\nabla^{\underline{\mathcal{T}}})\wedge(\ch_{q_2}\cdots\ch_{q_m})(\nabla^{\underline{\mathcal{T}}}).
    \label{eq:CSinvariant_ch}
\end{equation}
Moreover, we use the coframe $(\eta^i = dw^i)$ instead of $(\theta^\alpha, \sqrt{-1}\theta)$, since $\ch_q(\nabla^{\bm{s}},\nabla^{\underline{\mathcal{T}}})$ is then given by 
\begin{equation}
    \frac{1}{(q-1)!}\left(\frac{\sqrt{-1}}{2\pi}\right)^q\int_0^1\tr\left((\theta_{i}^{\,\,\,j})\wedge(t(\Theta_{i}^{\,\,\,j})+(t^2-t)(\theta_i^{\,\,\,j})\wedge(\theta_i^{\,\,\,j}))^{q-1}\right)dt.
    \label{eq:CSform_ch}
\end{equation}
\begin{lem}
    For the boundary $M_r$ of the Reinhardt domain $\Omega_r$, the form \eqref{eq:CSform_ch} is a sum of forms of the form $\eta^{n+1}\wedge\alpha_{q-1+k,q-1-k}$, where $k\ge0$ and $\alpha_{q-1+k,q-1-k}$ is of bidegree $(q-1+k,q-1-k)$.
\end{lem}
\begin{proof}
    The connection form with respect to the coframe $(\eta^i)$ is given by 
    \begin{equation}
        (\theta_i^{\,\,\,j}) = \frac{1}{2r(n+1)}\begin{pmatrix}
            -\eta^{n+1}\delta_{\beta}^{\,\,\,\alpha}& -\eta^\alpha\\
            (n+1)\eta^\beta& n\eta^{n+1}
        \end{pmatrix}.
    \end{equation}
    Moreover, setting $I_n = {\rm diag}(1,\ldots,1,0)$ and $\eta = (\eta^1,\ldots,\eta^n,0)^{\top}$, the curvature form satisfies 
    \begin{equation}
    \scalebox{0.95}{$\displaystyle
        (\Theta_{i}^{\,\,\,j}) = \frac{1}{4r^2(n+1)}\left((\eta^\top \wedge \conj{\eta})I_n + \eta \wedge \conj{\eta}^\top + (n+1)(\eta+\conj{\eta}) \wedge \eta^\top\right) \mod \eta^{n+1}.
        $}
    \end{equation}
    A bidegree count shows that the form \eqref{eq:CSform_ch} vanishes modulo $\eta^{n+1}$.
    The components of $(\theta_i^{\,\,\,j})$ and $(\Theta_i^{\,\,\,j})$ are all $(i,j)$-forms with $i\ge j$, possibly multiplied by $\eta^{n+1}$, except for the $(\alpha,n+1)$-component of the curvature 
    \begin{equation}
        \Theta_{n+1}^{\,\,\,\alpha} = -\frac{1}{4r^2}\eta^{n+1}\wedge\eta^\alpha - \frac{n+2}{4r^2(n+1)}\eta^{n+1}\wedge\eta^{\conj{\alpha}}.
    \end{equation}
    However, in the expansion of \eqref{eq:CSform_ch}, the exceptional components must be multiplied by the $(n+1,\alpha)$-components of the connection $\theta_{\alpha}^{\,\,\,n+1} = \frac{1}{2r}\eta^\alpha$.
    This proves the lemma.
\end{proof}

\begin{lem}
    Let 
    \begin{equation}
        (\Theta_i^{\,\,\,j})^{(1,1)} = \frac{1}{4r^2(n+1)}\left((\eta^\top \wedge \conj{\eta})I_n + \eta \wedge \conj{\eta}^\top + (n+1)\conj{\eta} \wedge \eta^\top\right)
    \end{equation}
    denote the $(1,1)$-part of $\Theta$.
    The integral \eqref{eq:CSinvariant_ch} is equal to 
    \begin{equation}
    \scalebox{0.95}{$
    \begin{aligned}
        &\frac{1}{q_1!\cdots q_m!}\left(\frac{\sqrt{-1}}{2\pi}\right)^{n+1} \int_{M_r} \tr\left((\theta_i^{\,\,\,j})\wedge(\Theta_i^{\,\,\,j})^{q_1-1}\right)\\
        &\qquad\qquad\qquad\qquad\qquad\qquad\wedge\tr\left(((\Theta_i^{\,\,\,j})^{(1,1)})^{q_2}\right)\wedge\cdots\wedge\tr\left(((\Theta_i^{\,\,\,j})^{(1,1)})^{q_m}\right).
    \end{aligned}
    $}
    \end{equation}
\end{lem}
\begin{proof}
The previous lemma implies that, in evaluating the integral \eqref{eq:CSinvariant_ch}, we may ignore the terms in $(\ch_{q_2}\cdots\ch_{q_m})(\Theta_i^{\,\,\,j})$ involving $\eta^{n+1}$.
If we ignore the components in $\Theta_i^{\,\,\,j}$ involving $\eta^{n+1}$, then all the components are sums of forms with bidegree $(2,0)$ and $(1,1)$.
Therefore, the only terms that contribute to \eqref{eq:CSinvariant_ch} are the $\eta^{n+1}\wedge((q_1-1,q_1-1)\text{-form})$-part of $\ch_{q_1}(\nabla^{\bm{s}},\nabla^{\underline{\mathcal{T}}})$ and the $(n+1-q_1,n+1-q_1)$-part of $(\ch_{q_2}\cdots\ch_{q_m})(\nabla^{\underline{\mathcal{T}}})$.

The former is equal to 
\begin{equation}
\begin{aligned}
    &\frac{1}{(q_1-1)!}\left(\frac{\sqrt{-1}}{2\pi}\right)^{q_1} \int_0^1 \tr((\theta_i^{\,\,\,j}) \wedge (t\Theta_i^{\,\,\,j})^{q_1-1})dt \\
    &\qquad = \frac{1}{q_1!}\left(\frac{\sqrt{-1}}{2\pi}\right)^{q_1} \tr((\theta_i^{\,\,\,j})\wedge(\Theta_i^{\,\,\,j})^{q_1-1}),
\end{aligned}
\end{equation}
because, in the binomial expansion of $t(\Theta_{i}^{\,\,\,j})+(t^2-t)(\theta_i^{\,\,\,j})\wedge(\theta_i^{\,\,\,j})$ to the power $q_1-1$, the terms involving $(\theta_i^{\,\,\,j})\wedge(\theta_i^{\,\,\,j})$ do not contribute to the $\eta^{n+1}\wedge((q_1-1,q_1-1)\text{-form})$-part.
The latter is easily seen to be equal to $(\ch_{q_2}\cdots\ch_{q_m})((\Theta_i^{\,\,\,j})^{(1,1)})$.
This proves the lemma.
\end{proof}

\begin{lem}
    For every $k\ge1$, we have 
    \begin{equation}
        \frac{1}{k!}\tr\left(((\Theta_i^{\,\,\,j})^{(1,1)})^k\right) = \frac{C_k}{(2r)^{2k}}(\eta^\top \wedge \conj{\eta})^k,
    \end{equation}
    where 
    \begin{equation}
        C_k = \frac{(n+2)-(n+2)^k}{k!(n+1)^k}.
    \end{equation}
\end{lem}
\begin{proof}
    From the definition of $(\Theta_i^{\,\,\,j})^{(1,1)}$, a direct calculation gives 
    \begin{equation}
    \begin{aligned}
        \left((\Theta_i^{\,\,\,j})^{(1,1)}\right)^k &= \frac{1}{(2r)^{2k}(n+1)^k}(\eta^\top \wedge \conj{\eta})^{k-1}\\ &\qquad\wedge \left((\eta^\top \wedge \conj{\eta})I_n + \eta \wedge \conj{\eta}^\top + ((n+2)^k-1)\conj{\eta} \wedge \eta^\top \right).
    \end{aligned}
    \end{equation}
    Taking trace, we obtain the desired result.
\end{proof}

We are left with the computation of the $\eta^{n+1}\wedge((q_1-1,q_1-1)\text{-form})$-part of $\tr\left((\theta_i^{\,\,\,j})\wedge(\Theta_i^{\,\,\,j})^{q_1-1}\right)$.
The computation splits into the cases $q_1=2$ and $q_1\ge 3$.

\begin{prop}
    The integral 
    \begin{equation}
        \int_{M_r} \ch_2(\nabla^{\bm{s}},\nabla^{\underline{\mathcal{T}}}) \wedge (\ch_{q_2}\cdots\ch_{q_m})(\nabla^{\underline{\mathcal{T}}})
    \end{equation}
    is given by 
    \begin{equation}
        \frac{n!C_2C_{q_2}\cdots C_{q_m}}{(2r)^{n+1}}\vol(\mathbb{S}^{n}).
    \end{equation}
\end{prop}

\begin{proof}
    The $\eta^{n+1}\wedge((1,1)\text{-form})$-part of $\tr((\theta_i^{\,\,\,j})\wedge(\Theta_i^{\,\,\,j}))$ arises from the product of an $(\alpha,\alpha)$-component of $(\theta_i^{\,\,\,j})$ with an $(\alpha,\alpha)$-component of $(\Theta_i^{\,\,\,j})$, and from the product of an $(n+1,\alpha)$-component of $(\theta_i^{\,\,\,j})$ with an $(\alpha,n+1)$-component of $(\Theta_i^{\,\,\,j})$.
    However, since all $(\alpha,\alpha)$-components of $(\theta_i^{\,\,\,j})$ are the same and  $(\Theta_{i}^{\,\,\,j})$ is trace-free, the former terms sum to zero.
    Therefore, the $\eta^{n+1}\wedge((1,1)\text{-form})$-part of $\tr\left((\theta_i^{\,\,\,j})\wedge(\Theta_i^{\,\,\,j})\right)$ is given by
    \begin{equation}
        -\frac{2!C_2}{(2r)^3}\eta^{n+1}\wedge(\eta^\top \wedge \conj{\eta})
    \end{equation}
    and hence 
    \begin{align}
        & \ch_2(\nabla^{\bm{s}},\nabla^{\underline{\mathcal{T}}}) \wedge (\ch_{q_2}\cdots\ch_{q_m})(\nabla^{\underline{\mathcal{T}}})\\
        &= -\frac{\sqrt{-1}^{n+1}C_2C_{q_2}\cdots C_{q_m}}{(2\pi)^{n+1}(2r)^{2n+1}}\eta^{n+1}\wedge(\delta_{\alpha\conj{\beta}}\eta^\alpha \wedge \eta^{\conj{\beta}})^n\\
        &= \frac{n!C_2C_{q_2}\cdots C_{q_m}}{2\cdot\pi^{n+1}(2r)^{2n+1}}dy^{n+1}\wedge\left(\wedge_{\alpha=1}^n dx^\alpha \wedge dy^\alpha\right).
    \end{align}
    Integrating along the fibers of $\pi\colon M_r \to \mathbb{S}_r^n$, we obtain 
    \begin{align}
        &\int_{M_r} \ch_2(\nabla^{\bm{s}},\nabla^{\underline{\mathcal{T}}}) \wedge (\ch_{q_2}\cdots\ch_{q_m})(\nabla^{\underline{\mathcal{T}}})\\
        &= \frac{2^{n}n!C_2C_{q_2}\cdots C_{q_m}}{(2r)^{2n+1}}\vol(\mathbb{S}^{n}_r)\\
        &= \frac{n!C_2C_{q_2}\cdots C_{q_m}}{(2r)^{n+1}}\vol(\mathbb{S}^{n})
    \end{align}
    as desired.
\end{proof}

Since the connection form $(\theta_i^{\,\,\,j})$ of $\nabla^{\underline{\mathcal{T}}}$ with respect to the coframe $(\eta^i)$ is trace-free, we have
\begin{equation}
    c_2(\nabla^{\bm{s}},\nabla^{\underline{\mathcal{T}}}) = -\ch_2(\nabla^{\bm{s}},\nabla^{\underline{\mathcal{T}}}).
\end{equation}
Therefore, we obtain the following corollary.
\begin{cor}
    The integral
    \begin{equation}
        \int_{M_r} c_2(\nabla^{\bm{s}},\nabla^{\underline{\mathcal{T}}}) \wedge (\ch_{q_2}\cdots\ch_{q_m})(\nabla^{\underline{\mathcal{T}}})
    \end{equation}
    is given by 
    \begin{equation}
        -\frac{n!C_2C_{q_2}\cdots C_{q_m}}{(2r)^{n+1}}\vol(\mathbb{S}^{n}).
    \end{equation}
\end{cor}

\begin{ex}
    Let us consider the case $n=3$. 
    Since the curvature form $\Theta_i^{\,\,\,j}$ is trace-free, we have $c_2(\nabla^{\underline{\mathcal{T}}}) = -\ch_2(\nabla^{\underline{\mathcal{T}}})$. Hence 
    \begin{equation}
        \widetilde{\mu}_{c_2\cdot c_2}(M_r) = \frac{75\pi^2}{256r^4}.
    \end{equation}
\end{ex}

We next treat the case $q_1\ge 3$.
\begin{prop}
    When $q_1\ge 3$, the integral 
    \begin{equation}
        \int_{M_r} \ch_{q_1}(\nabla^{\bm{s}},\nabla^{\underline{\mathcal{T}}}) \wedge (\ch_{q_2}\cdots\ch_{q_m})(\nabla^{\underline{\mathcal{T}}})
    \end{equation}
    is given by 
    \begin{equation}
        \frac{n!\widetilde{C}_{q_1}C_{q_2}\cdots C_{q_m}}{(2r)^{n+1}}\vol(\mathbb{S}^{n}),
    \end{equation}
    where 
    \begin{equation}
        \widetilde{C}_{k} = \frac{(n+2)-(n+2)^{k-1}}{k!(n+1)^k}.
    \end{equation}
\end{prop}

\begin{proof}
    As in the case of $q_1=2$, we consider the $\eta^{n+1}\wedge((q_1-1,q_1-1)\text{-form})$-part of $\tr\left((\theta_i^{\,\,\,j})\wedge(\Theta_i^{\,\,\,j})^{q_1-1}\right)$.
    These terms can arise only from the product of an $(\alpha,\alpha)$-component of $(\theta_i^{\,\,\,j})$ with an $(\alpha,\alpha)$-component of $((\Theta_i^{\,\,\,j})^{(1,1)})^{q_1-1}$, and their sum is 
    \begin{equation}
        -\frac{q_1! \widetilde{C}_{q_1}}{(2r)^{2q_1-1}}\eta^{n+1}\wedge(\eta^\top \wedge \conj{\eta})^{q_1-1}.
    \end{equation}
    The remaining part of the proof follows by a similar argument.
\end{proof}

\subsection{Sasakian \texorpdfstring{$\eta$}{eta}-Einstein manifolds}
In this subsection, we compute the CR Cheeger--Simons invariant for regular Sasakian $\eta$-Einstein manifolds.
Our computation is based on the ambient construction for such manifolds due to Takeuchi \cite{takeuchi_ambient_2018}.

Let $Y$ be a compact complex manifold of complex dimension $n$, and let $\pi\colon (L,h)\to Y$ be a negative Hermitian holomorphic line bundle over $Y$.
Its curvature $\Theta_h$ defines a Kähler form on $Y$ by 
\begin{equation}
    \omega = -\sqrt{-1}\Theta_h = \sqrt{-1}\partial\conj{\partial}\log h.
\end{equation}

Let $M$ be the associated $\mathit{U}(1)$-bundle of $(L,h)$. 
It carries a natural CR structure as the boundary of the disk bundle $\Omega = \{v\in L \mid \lVert v \rVert_h < 1\}$.
The function $\log h$ is a defining function of $\Omega$, and the associated contact form
\begin{equation}
    \theta = -\sqrt{-1}\partial \log h|_{TM}
\end{equation}
is a constant multiple of the connection form of $(L,h)$.
It follows that the bundle $T^{1,0}M$ is the horizontal lift of $T^{1,0}Y$.
Since $d\theta = \pi^*\omega$, the Levi form $l_\theta$ is the pullback of the Kähler metric, and so $M$ is strictly pseudoconvex as a CR manifold.
It is known that the Tanaka--Webster connection on $T^{1,0}M$ associated with the contact form $\theta$ is the lift of the Chern connection of $\omega$ and hence the torsion $A_{\alpha\beta}$ vanishes.
In general, a strictly pseudoconvex CR manifold equipped with a contact form whose Tanaka--Webster torsion vanishes is called a \emph{Sasakian manifold}; see \cite{boyer_sasakian_2008} for Sasakian geometry.
The punctured disk bundle $\Omega\setminus Y$ is biholomorphic to the truncated Kähler cone $(0,1)\times M$ via  
\begin{equation}
    \Omega \setminus Y \ni v \mapsto \left(\lVert v \rVert_h,\frac{v}{\lVert v \rVert_h}\right) \in C(M).
\end{equation}

We assume that $\omega$ is Einstein with Einstein constant $(n+1)\lambda$.
Then the formula \eqref{eq:pseudo_Einstein} is satisfied and $\theta$ is a pseudo-Einstein contact form.
The pseudo-Einstein structure constructed in this way is called the \emph{regular Sasakian $\eta$-Einstein structure}.
Takeuchi \cite{takeuchi_ambient_2018} showed that 
\begin{equation}
    \rho = \begin{cases}
        \lambda^{-1}(h^\lambda - 1),& \lambda \ne 0,\\
        \log h,& \lambda = 0,
    \end{cases}
\end{equation}
is a Fefferman defining function of $\Omega$.
From 
\begin{align}
    -\sqrt{-1}\partial \rho &= (1+\lambda\rho)\theta,\\
    \sqrt{-1}\partial\conj{\partial}\rho &= (1+\lambda\rho)\pi^*\omega + \sqrt{-1}\lambda(1+\lambda\rho)^{-1}\partial \rho \wedge \conj{\partial} \rho,
\end{align}
we see that the two defining functions $\log h$ and $\rho$ determine the same contact form $\theta$, whereas their transverse curvatures are $0$ and $\lambda(1+\lambda\rho)^{-1}$, respectively.
This observation is fundamental for the calculations that follow.

\subsubsection{Computation of renormalized characteristic numbers}
We compute the integral 
\begin{equation}
    \int_\Omega (c_{q_1}\cdots c_{q_m})(\conj{\nabla}^g) \quad (q_1+\cdots +q_m = n+1),
\end{equation}
where $\conj{\nabla}^g$ is the renormalized connection of $\rho$.
To this end, we also use the renormalized connection $\conj{\nabla}^{\widetilde{g}}$ associated with the defining function $\log h$, where
\begin{equation}
    \widetilde{g} = -\sqrt{-1}\partial\conj{\partial}\log(-\log h).
\end{equation} 
The transgression formula \eqref{eq:difference_form} gives 
\begin{equation}
    \int_\Omega (c_{q_1}\cdots c_{q_m})(\conj{\nabla}^g) = \int_{\Omega} (c_{q_1}\cdots c_{q_m})(\conj{\nabla}^{\widetilde{g}}) + \int_M  (c_{q_1}\cdots c_{q_m})(\conj{\nabla}^{\widetilde{g}},\conj{\nabla}^g).
\end{equation}

Let $\theta^\alpha$ be a coframe of $T^{1,0}Y$, and let $\pi_{\beta}^{\,\,\,\alpha}$ and $\Pi_{\beta}^{\,\,\,\alpha}$ be the connection and curvature forms of the Chern connection of $\omega$, respectively.
By \cite[Proposition 3.5]{marugame_renormalized_2016}, the connection form of $\conj{\nabla}^{\widetilde{g}}$ with respect to the coframe $(\theta^\alpha,\theta^{n+1}=\partial\log h)$ is given near the boundary by
\begin{equation}
    \begin{pmatrix}
        \pi_{\beta}^{\,\,\,\alpha}& 0\\
        -\theta_\beta& 0
    \end{pmatrix}.
    \label{eq:renormalized_connection_logh}
\end{equation}
We choose a smooth extension of this connection to all of $\Omega$ whose connection form on $\Omega\setminus Y$ has the same diagonal and upper-right components as above, while the lower-left component vanishes in a neighborhood of the zero section $Y$.
Then the curvature form is given by 
\begin{equation}
    \begin{pmatrix}
        \Pi_{\beta}^{\,\,\,\alpha}& 0\\
        \bullet& 0
    \end{pmatrix}.
\end{equation}
Hence 
\begin{equation}
    (c_{q_1}\cdots c_{q_m})(\conj{\nabla}^{\widetilde{g}}) = (c_{q_1}\cdots c_{q_m})(\Pi_{\beta}^{\,\,\,\alpha}) = 0
\end{equation}
for dimensional reasons.
Therefore, 
\begin{equation}
    \int_{\Omega} (c_{q_1}\cdots c_{q_m})(\conj{\nabla}^{\widetilde{g}}) = 0.
\end{equation}

It remains to calculate the integral of the relative Chern--Simons form 
\begin{equation}
    \int_M  (c_{q_1}\cdots c_{q_m})(\conj{\nabla}^{\widetilde{g}},\conj{\nabla}^g).
\end{equation}
We use the Chern characters instead of the Chern polynomials.
\begin{prop}
    For $q\ge1$, set
    \begin{equation}
    \begin{aligned}
        \mathcal{A}_q(t)&=\sum_{i=0}^{q}
        \frac{1}{(q-i)!}(\lambda t c_1(L))^{q-i}\wedge\ch_i(T^{1,0}Y\oplus\underline{\mathbb{C}}),\\
        \mathcal{B}_q(t)&=\sum_{i=0}^{q-1}\frac{1}{(q-1-i)!}(\lambda t c_1(L))^{q-1-i}\wedge\ch_i(T^{1,0}Y\oplus\underline{\mathbb{C}}).
    \end{aligned}
    \end{equation}
    The integral of the relative Chern--Simons form
    \begin{equation}
        \int_M(\ch_{q_1}\cdots\ch_{q_m})(\conj{\nabla}^{\widetilde{g}},\conj{\nabla}^g)
    \end{equation}
    is given by
    \begin{equation}
        -\lambda\sum_{k=1}^{m}\int_0^1\int_Y\mathcal{B}_{q_k}(t)\wedge\bigwedge_{j\ne k}\mathcal{A}_{q_j}(t)\,dt.
    \end{equation}
\end{prop}

\begin{proof}
    By \cite[Proposition 3.5]{marugame_renormalized_2016}, the connection form of $\conj{\nabla}^g$ with respect to the coframe $(\theta^\alpha,\theta^{n+1} = \partial \log h|_M = \partial \rho|_M)$ is given by 
    \begin{equation}
    \begin{pmatrix}
        \pi_{\beta}^{\,\,\,\alpha}+\sqrt{-1}\lambda\theta\delta_{\beta}^{\,\,\,\alpha}& \lambda\theta^{\alpha}\\
        -\theta_\beta& \sqrt{-1}\lambda\theta
    \end{pmatrix}.
    \end{equation}
    Using this and \eqref{eq:renormalized_connection_logh}, the curvature form of the homotopy of connections $(1-t)\conj{\nabla}^{\widetilde{g}}+t\conj{\nabla}^g$ becomes 
    \begin{equation}
        \begin{pmatrix}
            \Pi_{\beta}^{\,\,\,\alpha}-tC_{\beta}^{\,\,\,\alpha}& 0\\
            0& 0
        \end{pmatrix},
        \label{eq:curvature_of_homotopy}
    \end{equation}
    where $C_{\beta}^{\,\,\,\alpha} = \lambda(\delta_{\beta}^{\,\,\,\alpha}l_{\gamma\conj{\delta}} + l_{\beta\conj{\delta}}\delta_{\gamma}^{\,\,\,\alpha})\theta^\gamma \wedge \theta^{\conj{\delta}}$.

    Now the polarization $\widetilde{\ch_{q_1}\cdots\ch_{q_m}}$ satisfies 
    \begin{small}
    \begin{equation}
        \widetilde{\ch_{q_1}\cdots\ch_{q_m}}(A,B,\ldots,B) = \sum_{k=1}^{m}\frac{q_k}{n+1}\ch_{q_1}(B)\cdots\widetilde{\ch}_{q_k}(A,B,\ldots,B)\cdots\ch_{q_m}(B)
    \end{equation}
    \end{small}
    and $\widetilde{\ch}_{q_k}$ is given by 
    \begin{equation}
        \widetilde{\ch}_{q_k}(A,B,\ldots,B) = \frac{1}{q_k!}\left(\frac{\sqrt{-1}}{2\pi}\right)^{q_k} \tr(AB^{q_k-1}).
    \end{equation}
    Using 
    \begin{gather}
        \tr C = -\sqrt{-1}(n+1)\lambda\omega,\\
        \Pi \wedge C = C \wedge \Pi = -\sqrt{-1}\lambda\omega \wedge \Pi,\quad C \wedge C = -\sqrt{-1}\lambda\omega \wedge C,
    \end{gather}
    one computes
    \begin{align}
        &\widetilde{\ch}_{q}\left((\conn(\conj{\nabla}^g)-\conn(\conj{\nabla}^{\widetilde{g}}))\wedge\curv((1-t)\conj{\nabla}^{\widetilde{g}}+t\conj{\nabla}^g)^{q-1}\right)\\
        &= -\frac{\lambda}{2\pi}\theta \wedge \frac{1}{q}\sum_{i=0}^{q-1}\frac{1}{(q-1-i)!}(\lambda tc_1(\nabla^h))^{q-1-i}\wedge\ch_i(\nabla^\omega\oplus\nabla^{\underline{\mathbb{C}}})
    \end{align}
    and 
    \begin{align}
        &\ch_q\left(\curv((1-t)\conj{\nabla}^{\widetilde{g}}+t\conj{\nabla}^g)\right)\\
        &= \sum_{i=0}^{q}\frac{1}{(q-i)!}(\lambda tc_1(\nabla^h))^{q-i}\wedge\ch_i(\nabla^\omega \oplus \nabla^{\underline{\mathbb{C}}}),
    \end{align}
    where $\nabla^h$ is the Chern connection of $h$, $\nabla^\omega$ is the Chern connection of $\omega$, and $\nabla^{\underline{\mathbb{C}}}$ denotes the trivial connection on the trivial line bundle $\underline{\mathbb{C}}$.
    Substituting these into the formula \eqref{eq:relativeCS} for the relative Chern--Simons form and integrating over $M$, we obtain the stated formula.
\end{proof}

Once these renormalized characteristic numbers are computed, the value of $\widetilde{\mu}_{c_{n+1}}(M)$ can be recovered from the renormalized Chern--Gauss--Bonnet formula \eqref{eq:renormalized_CGB} together with $\chi(\Omega)=\chi(Y)$.
We remark that another expression for $\widetilde{\mu}_{c_{n+1}}(M)$ is given in \cite{takeuchi_formulae_2024}.

\begin{ex}
    Assume $\dim_{\mathbb{C}}Y = 3$.
    Since $c_2\cdot c_2 = \frac{1}{4}\ch_1^4 - \ch_1^2\cdot\ch_2 + \ch_2^2$, the CR Cheeger--Simons invariant associated with $c_2\cdot c_2$ is given by
    \begin{equation}
        \mu_{c_2\cdot c_2}(M) = -\frac{3}{4}\lambda\int_Y c_1(T^{1,0}Y)^3 + \mathbb{Z}.
    \end{equation}
\end{ex}

\subsubsection{CR Cheeger--Simons invariant associated with $c_1\cdot\Phi$}
In this subsection, we derive a convenient formula for the CR Cheeger--Simons invariant associated with an invariant polynomial of the form $c_1\cdot\Phi$. 
One motivation is that the non-integrality of the corresponding real number gives an obstruction to CR embeddability into $\mathbb{C}^{n+1}$; see Proposition \ref{thm:obstruction_to_embedding}. 
However, the examples presented below do not yield new results concerning this embedding problem.

\begin{prop}
    For $\Phi = c_{q_2}\cdots c_{q_m}$ with $q_2+\cdots+q_m=n$, we have 
    \begin{equation}
        \mu_{c_1\cdot\Phi}(M) = -(n+1)\lambda\int_Y\Phi(T^{1,0}Y) + \mathbb{Z}.
    \end{equation}
\end{prop}

\begin{proof}
    By the discussion in the previous subsection, we have 
    \begin{equation}
        \mu_{c_1\cdot\Phi}(M) = \int_M (c_1\cdot\Phi)(\conj{\nabla}^{\widetilde{g}},\conj{\nabla}^g)+\mathbb{Z}.
    \end{equation}
    By \eqref{eq:productCS}, the right-hand side is equal to 
    \begin{equation}
        \int_M \left(c_1(\conj{\nabla}^{\widetilde{g}},\conj{\nabla}^g) \wedge \Phi(\conj{\nabla}^{\widetilde{g}})+c_1(\conj{\nabla}^g) \wedge \Phi(\conj{\nabla}^{\widetilde{g}},\conj{\nabla}^g)\right) +\mathbb{Z}.
    \end{equation}
    Using \eqref{eq:curvature_of_homotopy}, one computes 
    \begin{equation}
        c_1(\conj{\nabla}^{\widetilde{g}},\conj{\nabla}^g) = \frac{\sqrt{-1}}{2\pi}\tr\begin{pmatrix}
            \sqrt{-1}\lambda\theta\delta_{\beta}^{\,\,\,\alpha}& \lambda\theta^\alpha\\
            0& \sqrt{-1}\lambda\theta
        \end{pmatrix}
        = -\frac{n+1}{2\pi}\lambda\theta,
    \end{equation}
    and
    \begin{equation}
        c_1(\conj{\nabla}^g) = c_1(\Pi-C) = \frac{\sqrt{-1}}{2\pi}(\Ric(\omega)+\sqrt{-1}(n+1)\lambda\omega) = 0.
    \end{equation}
    Moreover, $\Phi(\conj{\nabla}^{\widetilde{g}})=\Phi(\Pi)$. 
    Since integration along each $\mathit{U}(1)$-fiber gives $\int_{\pi^{-1}(y)}\theta=2\pi$, the desired formula follows.
\end{proof}

\begin{ex}
    Let $M$ be the associated $\mathit{U}(1)$-bundle of 
    \begin{equation}
    \left((\det U)^{\otimes m},(\det h_{\rm FS})^{\otimes m}\right)\to\mathrm{Gr}(k,\mathbb{C}^l), \quad k(l-k)=n,
    \end{equation}
    where $U$ is the universal bundle over the Grassmannian and $h_{\rm FS}$ is the Fubini--Study metric.
    A result of \cite{arezzo_szego_2013} shows that $M$ is not spherical when $k>1$.
    The Einstein constant $\lambda$ can be computed as follows.
    First, we have $c_1((\det U)^{\otimes m}) = -\frac{1}{2\pi}\omega$, so by the Einstein condition $\Ric(\omega) = -\sqrt{-1}(n+1)\lambda\omega$, 
    \begin{equation}
        c_1(T^{1,0}\mathrm{Gr}(k,\mathbb{C}^l)) = \frac{n+1}{2\pi}\lambda\omega = -(n+1)\lambda c_1((\det U)^{\otimes m}).
    \end{equation}
    On the other hand, since $T^{1,0}\mathrm{Gr}(k,\mathbb{C}^l) = U^*\otimes(\underline{\mathbb{C}}^l/U)$, we have 
    \begin{equation}
    \begin{aligned}
        c_1(T^{1,0}\mathrm{Gr}(k,\mathbb{C}^l)) &= (l-k)c_1(U^*) + kc_1(\underline{\mathbb{C}}^l/U)\\
        &= -lc_1(U) = -\frac{l}{m}c_1((\det U)^{\otimes m}).
    \end{aligned}
    \end{equation}
    Comparing these two formulas, we obtain $\lambda = \frac{l}{(n+1)m}$.
    Consequently, we have 
    \begin{equation}
        \mu_{c_1\cdot\Phi}(M) = -\frac{l}{m}\int_{\mathrm{Gr}(k,\mathbb{C}^l)}\Phi(T^{1,0}\mathrm{Gr}(k,\mathbb{C}^l)) + \mathbb{Z}.
    \end{equation}

    For example, when $k=1$ and $l=2$, $M$ is the lens space $\mathbb{S}^3/\mathbb{Z}_m$ with the standard CR structure.
    It follows that 
    \begin{equation}
        \mu_{c_1\cdot c_1}(\mathbb{S}^3/\mathbb{Z}_m) = -\frac{4}{m}+\mathbb{Z}.
    \end{equation}
    Hence $\mathbb{S}^3/\mathbb{Z}_m$ admits no CR embedding into $\mathbb{C}^2$ whenever $m\notin\{1,2,4\}$. We remark that the remaining cases $m=2,4$ are also non-embeddable.
\end{ex}

\begin{rem}
    As a corollary of the computations in this section, we obtain $\mu_{c_1\cdot c_1} = 4\mu_{c_2}$ for regular Sasakian $\eta$-Einstein manifolds.
    On the other hand, we have the following example where the relation does not hold.

    Consider the family of CR structures on $\mathbb{S}^3 = \{(z,w)\in \mathbb{C}^2 \mid \abs{z}^2 + \abs{w}^2 = 1\}$ given by 
    \begin{equation}
        T^{1,0}_t M = \left\langle Z_1 + \frac{t}{\sqrt{1+t^2}}Z_{\conj{1}} \right\rangle, \quad Z_1 = \conj{w}\frac{\partial}{\partial z} - \conj{z}\frac{\partial}{\partial w}.
    \end{equation}
    The CR manifold $\mathbb{S}_t^3 = (\mathbb{S}^3,T_t^{1,0} \mathbb{S}^3)$ is called a \emph{Rossi sphere} \cite{rossi_attaching_1965}.
    In \cite{burns_global_1988}, Burns and Epstein computed 
    \begin{equation}
        \widetilde{\mu}_{c_2}^{\rm BE}(\mathbb{S}^3_t) = -1 + 12t^2(1+t^2).
    \end{equation}
    On the other hand, since the curvature of $\nabla^{\mathcal{T}}$ is trace-free, we have $\widetilde{\mu}_{c_1\cdot c_1}(\mathbb{S}^3_t) = 0$.
    Therefore, 
    \begin{equation}
        \mu_{c_1\cdot c_1}(\mathbb{S}^3_t) = 0+\mathbb{Z} \ne 48t^2(1+t^2)+\mathbb{Z} = 4\mu_{c_2}(\mathbb{S}^3_t)
    \end{equation}
    for generic $t$.
\end{rem}

\bibliographystyle{abbrv}
\bibliography{bibibi}

@article{marugame_renormalized_2016,
	title = {Renormalized {Chern}-{Gauss}-{Bonnet} formula for complete {Kähler}-{Einstein} metrics},
	volume = {138},
	issn = {0002-9327,1080-6377},
	doi = {10.1353/ajm.2016.0034},
	number = {4},
	journal = {Amer. J. Math.},
	author = {Marugame, Taiji},
	year = {2016},
	mrnumber = {3538151},
	pages = {1067--1094},
}

@article{marugame_renormalized_2021,
	title = {Renormalized characteristic forms of the {Cheng}-{Yau} metric and global {CR} invariants},
	volume = {377},
	issn = {0001-8708,1090-2082},
	doi = {10.1016/j.aim.2020.107468},
	journal = {Adv. Math.},
	author = {Marugame, Taiji},
	year = {2021},
	mrnumber = {4186011},
	pages = {Paper No. 107468, 55},
}

@article{burns_global_1988,
	title = {A global invariant for three-dimensional {CR}-manifolds},
	volume = {92},
	issn = {0020-9910,1432-1297},
	doi = {10.1007/BF01404456},
	number = {2},
	journal = {Invent. Math.},
	author = {Burns, Jr., Daniel M. and Epstein, Charles L.},
	year = {1988},
	mrnumber = {936085},
	pages = {333--348},
}

@article{burns_characteristic_1990,
	title = {Characteristic numbers of bounded domains},
	volume = {164},
	issn = {0001-5962,1871-2509},
	doi = {10.1007/BF02392751},
	number = {1-2},
	journal = {Acta Math.},
	author = {Burns, Jr., Daniel M. and Epstein, Charles  L.},
	year = {1990},
	mrnumber = {1037597},
	pages = {29--71},
}

@article{fefferman_monge-ampere_1976,
	title = {Monge-{Ampère} equations, the {Bergman} kernel, and geometry of pseudoconvex domains},
	volume = {103},
	issn = {0003-486X},
	doi = {10.2307/1970945},
	number = {2},
	journal = {Ann. of Math. (2)},
	author = {Fefferman, Charles L.},
	year = {1976},
	mrnumber = {407320},
	pages = {395--416},
}

@article{hislop_cr-invariants_2006,
	title = {{CR}-invariants and the scattering operator for complex manifolds with {CR}-boundary},
	volume = {342},
	issn = {1631-073X,1778-3569},
	doi = {10.1016/j.crma.2006.03.003},
	number = {9},
	journal = {C. R. Math. Acad. Sci. Paris},
	author = {Hislop, Peter D. and Perry, Peter A. and Tang, Siu-Hung},
	year = {2006},
	mrnumber = {2225870},
	pages = {651--654},
}

@book{suwa_complex_2024,
	title = {Complex analytic geometry—from the localization viewpoint},
	isbn = {978-981-4374-70-5 978-981-4374-71-2 978-981-4704-29-8},
	doi = {10.1142/8324},
	publisher = {World Scientific Publishing Co. Pte. Ltd., Hackensack, NJ},
	author = {Suwa, Tatsuo},
	year = {2024},
	mrnumber = {4729609},
}

@incollection{suwa_residue_2008,
	series = {Contemp. {Math}.},
	title = {Residue theoretical approach to intersection theory},
	volume = {459},
	isbn = {978-0-8218-4497-7},
	doi = {10.1090/conm/459/08972},
	booktitle = {Real and complex singularities},
	publisher = {Amer. Math. Soc., Providence, RI},
	author = {Suwa, Tatsuo},
	year = {2008},
	mrnumber = {2444403},
	pages = {207--261},
}

@article{chern_characteristic_1974,
	title = {Characteristic forms and geometric invariants},
	volume = {99},
	issn = {0003-486X},
	doi = {10.2307/1971013},
	journal = {Ann. of Math. (2)},
	author = {Chern, Shiing-Shen and Simons, James},
	year = {1974},
	mrnumber = {353327},
	pages = {48--69},
}

@article{chern_real_1974,
	title = {Real hypersurfaces in complex manifolds},
	volume = {133},
	issn = {0001-5962,1871-2509},
	doi = {10.1007/BF02392146},
	journal = {Acta Math.},
	author = {Chern, Shiing-Shen and Moser, Jürgen K.},
	year = {1974},
	mrnumber = {425155},
	pages = {219--271},
}

@article{tanaka_non-degenerate_1976,
	title = {On non-degenerate real hypersurfaces, graded {Lie} algebras and {Cartan} connections},
	volume = {2},
	issn = {0289-2316},
	doi = {10.4099/math1924.2.131},
	number = {1},
	journal = {Japan. J. Math. (N.S.)},
	author = {Tanaka, Noboru},
	year = {1976},
	mrnumber = {589931},
	pages = {131--190},
}

@incollection{cheeger_differential_1985,
	series = {Lecture {Notes} in {Math}.},
	title = {Differential characters and geometric invariants},
	volume = {1167},
	isbn = {978-3-540-16053-3},
	doi = {10.1007/BFb0075216},
	booktitle = {Geometry and topology ({College} {Park}, {Md}., 1983/84)},
	publisher = {Springer, Berlin},
	author = {Cheeger, Jeff and Simons, James},
	year = {1985},
	mrnumber = {827262},
	pages = {50--80},
}

@article{case_p-operator_2020,
	title = {The {P}'-operator, the {Q}'-curvature, and the {CR} tractor calculus},
	volume = {20},
	issn = {0391-173X,2036-2145},
	number = {2},
	journal = {Ann. Sc. Norm. Super. Pisa Cl. Sci. (5)},
	author = {Case, Jeffrey S. and Gover, A. Rod},
	year = {2020},
	mrnumber = {4105911},
	pages = {565--618},
}

@article{hirachi_q-prime_2014,
	title = {Q-prime curvature on {CR} manifolds},
	volume = {33},
	issn = {0926-2245,1872-6984},
	doi = {10.1016/j.difgeo.2013.10.013},
	journal = {Differential Geom. Appl.},
	author = {Hirachi, Kengo},
	year = {2014},
	mrnumber = {3159959},
	pages = {213--245},
}

@article{lempert_algebraic_1995,
	title = {Algebraic approximations in analytic geometry},
	volume = {121},
	issn = {0020-9910,1432-1297},
	doi = {10.1007/BF01884302},
	number = {2},
	journal = {Invent. Math.},
	author = {Lempert, László},
	year = {1995},
	mrnumber = {1346210},
	pages = {335--353},
}

@misc{matsumoto_cr_2022,
	title = {The {CR} {Killing} operator and {Bernstein}-{Gelfand}-{Gelfand} construction in {CR} geometry},
	url = {http://arxiv.org/abs/2205.11022},
	doi = {10.48550/arXiv.2205.11022},
	urldate = {2025-06-21},
	publisher = {arXiv},
	author = {Matsumoto, Yoshihiko},
	month = may,
	year = {2022},
	note = {arXiv:2205.11022 [math]},
}

@article{hirachi_variation_2017,
	title = {Variation of total {Q}-prime curvature on {CR} manifolds},
	volume = {306},
	issn = {0001-8708,1090-2082},
	doi = {10.1016/j.aim.2016.11.005},
	journal = {Adv. Math.},
	author = {Hirachi, Kengo and Marugame, Taiji and Matsumoto, Yoshihiko},
	year = {2017},
	mrnumber = {3581332},
	pages = {1333--1376},
}

@article{lee_pseudo-einstein_1988,
	title = {Pseudo-{Einstein} structures on {CR} manifolds},
	volume = {110},
	issn = {0002-9327,1080-6377},
	doi = {10.2307/2374543},
	number = {1},
	journal = {Amer. J. Math.},
	author = {Lee, John M.},
	year = {1988},
	mrnumber = {926742},
	pages = {157--178},
}

@article{gover_cr_2005,
	title = {{CR} invariant powers of the sub-{Laplacian}},
	volume = {583},
	issn = {0075-4102,1435-5345},
	doi = {10.1515/crll.2005.2005.583.1},
	journal = {J. Reine Angew. Math.},
	author = {Gover, A. Rod and Graham, C. Robin},
	year = {2005},
	mrnumber = {2146851},
	pages = {1--27},
}

@article{webster_pseudo-hermitian_1978,
	title = {Pseudo-{Hermitian} structures on a real hypersurface},
	volume = {13},
	issn = {0022-040X,1945-743X},
	number = {1},
	journal = {J. Differential Geometry},
	author = {Webster, Sidney M.},
	year = {1978},
	mrnumber = {520599},
	pages = {25--41},
}

@article{lee_fefferman_1986,
	title = {The {Fefferman} metric and pseudo-{Hermitian} invariants},
	volume = {296},
	issn = {0002-9947,1088-6850},
	doi = {10.2307/2000582},
	number = {1},
	journal = {Trans. Amer. Math. Soc.},
	author = {Lee, John M.},
	year = {1986},
	mrnumber = {837820},
	pages = {411--429},
}

@article{takeuchi_formulae_2024,
	title = {Formulae of some global {CR} invariants for {Sasakian} eta-{Einstein} manifolds},
	volume = {32},
	issn = {1019-8385,1944-9992},
	doi = {10.4310/cag.241015022326},
	number = {3},
	journal = {Comm. Anal. Geom.},
	author = {Takeuchi, Yuya},
	year = {2024},
	mrnumber = {4828194},
	pages = {667--692},
}

@article{takeuchi_ambient_2018,
	title = {Ambient constructions for {Sasakian} eta-{Einstein} manifolds},
	volume = {328},
	issn = {0001-8708,1090-2082},
	doi = {10.1016/j.aim.2018.01.007},
	journal = {Adv. Math.},
	author = {Takeuchi, Yuya},
	year = {2018},
	mrnumber = {3771124},
	pages = {82--111},
}

@article{arezzo_szego_2013,
	title = {Szegö kernel, regular quantizations and spherical {CR}-structures},
	volume = {275},
	issn = {0025-5874,1432-1823},
	doi = {10.1007/s00209-013-1178-1},
	number = {3-4},
	journal = {Math. Z.},
	author = {Arezzo, Claudio and Loi, Andrea and Zuddas, Fabio},
	year = {2013},
	mrnumber = {3127055},
	pages = {1207--1216},
}

@incollection{rossi_attaching_1965,
	title = {Attaching analytic spaces to an analytic space along a pseudoconcave boundary},
	booktitle = {Proc. {Conf}. {Complex} {Analysis} ({Minneapolis}, 1964)},
	publisher = {Springer, Berlin-Heidelberg-New York},
	author = {Rossi, H.},
	year = {1965},
	mrnumber = {176106},
	pages = {242--256},
}

@article{case_paneitz-type_2013,
	title = {A {Paneitz}-type operator for {CR} pluriharmonic functions},
	volume = {8},
	issn = {2304-7909,2304-7895},
	number = {3},
	journal = {Bull. Inst. Math. Acad. Sin. (N.S.)},
	author = {Case, Jeffrey S. and Yang, Paul},
	year = {2013},
	mrnumber = {3135070},
	pages = {285--322},
}

@article{case_cal_2023,
	title = {{\textbackslash}{Cal} {I}'-curvatures in higher dimensions and the {Hirachi} conjecture},
	volume = {75},
	issn = {0025-5645,1881-1167},
	doi = {10.2969/jmsj/87718771},
	number = {1},
	journal = {J. Math. Soc. Japan},
	author = {Case, Jeffrey S. and Takeuchi, Yuya},
	year = {2023},
	mrnumber = {4539017},
	pages = {291--328},
}

@book{penrose_spinors_1984,
	series = {Cambridge {Monographs} on {Mathematical} {Physics}},
	title = {Spinors and space-time. {Vol}. 1},
	isbn = {978-0-521-24527-2},
	doi = {10.1017/CBO9780511564048},
	publisher = {Cambridge University Press, Cambridge},
	author = {Penrose, Roger and Rindler, Wolfgang},
	year = {1984},
	mrnumber = {776784},
}

@book{boyer_sasakian_2008,
	series = {Oxford {Mathematical} {Monographs}},
	title = {Sasakian geometry},
	isbn = {978-0-19-856495-9},
	publisher = {Oxford University Press, Oxford},
	author = {Boyer, Charles P. and Galicki, Krzysztof},
	year = {2008},
	mrnumber = {2382957},
}

@incollection{tanaka_graded_1966,
	title = {Graded {Lie} algebras and geometric structures},
	booktitle = {Proc. {U}.{S}.-{Japan} {Seminar} in {Differential} {Geometry} ({Kyoto}, 1965)},
	publisher = {Nippon Hyoronsha Co., Ltd., Tokyo},
	author = {Tanaka, Noboru},
	year = {1966},
	mrnumber = {222802},
	pages = {147--150},
}

@article{borel_sur_1953,
	title = {Sur la cohomologie des espaces fibrés principaux et des espaces homogènes de groupes de {Lie} compacts},
	volume = {57},
	issn = {0003-486X},
	doi = {10.2307/1969728},
	journal = {Ann. of Math. (2)},
	author = {Borel, Armand},
	year = {1953},
	mrnumber = {51508},
	pages = {115--207},
}

@book{cap_parabolic_2009,
	series = {Mathematical {Surveys} and {Monographs}},
	title = {Parabolic geometries. {I}},
	volume = {154},
	isbn = {978-0-8218-2681-2},
	doi = {10.1090/surv/154},
	publisher = {American Mathematical Society, Providence, RI},
	author = {{\v{C}}ap, Andreas and Slovák, Jan},
	year = {2009},
	mrnumber = {2532439},
	doi = {10.1090/surv/154},
}

@article{farris_intrinsic_1986,
	title = {An intrinsic construction of {Fefferman}'s {CR} metric},
	volume = {123},
	issn = {0030-8730,1945-5844},
	number = {1},
	journal = {Pacific J. Math.},
	author = {Farris, Frank A.},
	year = {1986},
	mrnumber = {834136},
	pages = {33--45},
}

@article{cheng_burns-epstein_1990,
	title = {The {Burns}-{Epstein} invariant and deformation of {CR} structures},
	volume = {60},
	issn = {0012-7094,1547-7398},
	doi = {10.1215/S0012-7094-90-06008-9},
	number = {1},
	journal = {Duke Math. J.},
	author = {Chêng, Jih Hsin and Lee, John M.},
	year = {1990},
	mrnumber = {1047122},
	pages = {221--254},
}
\end{document}